\documentclass[]{article}
\usepackage{geometry}
\usepackage{amsmath}
\usepackage{amssymb}

\usepackage{bigfoot}

\usepackage{enumerate}

\usepackage{fancyhdr}
\usepackage[hidelinks]{hyperref} 

\usepackage{amsthm} 
\newtheorem{theorem}{Theorem}
\numberwithin{theorem}{section}
\newtheorem*{theorem*}{Theorem}
\newtheorem{definition}[theorem]{Definition}
\newtheorem{cor}[theorem]{Corollary}
\newtheorem{proposition}[theorem]{Proposition}
\newtheorem{lemma}[theorem]{Lemma}

\newtheorem{remark}[theorem]{Remark}

\newtheorem{example}[theorem]{Example}

\newtheorem*{example*}{Example}

\newtheorem{innercustomgeneric}{\customgenericname}
\providecommand{\customgenericname}{}
\newcommand{\newcustomtheorem}[2]{%
	\newenvironment{#1}[1]
	{%
		\renewcommand\customgenericname{#2}%
		\renewcommand\theinnercustomgeneric{##1}%
		\innercustomgeneric
	}
	{\endinnercustomgeneric}
}

\newcustomtheorem{customthm}{Theorem}
\newcustomtheorem{customlemma}{Lemma}

\usepackage{mathrsfs}

\usepackage{xcolor} 

\usepackage{caption}
\usepackage{subcaption}
\usepackage{graphicx}

\usepackage{tikz}
\usetikzlibrary{matrix,decorations.pathreplacing, calc, positioning,fit}

\usepackage{algorithm}
\usepackage{algpseudocode} 

\usepackage{bbm}

\usepackage{bm}

\usepackage{mathtools}

\usepackage[toc,page]{appendix}

\usepackage{caption}

\usepackage{tikz}
\usetikzlibrary{fit,shapes,arrows,positioning}

\tikzset{decision/.style={diamond, draw, text width=4.5em, text badly centered, inner sep=0pt}}
\tikzset{inout/.style={ellipse, draw, text width=7em, text centered, rounded corners,
		minimum width=3.5cm}}
\tikzset{block/.style={rectangle, draw, text width=12em, text centered, rounded corners,
		minimum width=3.5cm}}
\tikzset{block1/.style={rectangle, draw,fill=gray!20, text width=10em, text centered, rounded corners,
		minimum width=3.5cm}}
\tikzset{line/.style={draw, -latex}}

\title{Numerical integration of Schr\"odinger maps\\ via the Hasimoto transform}
\author{Valeria Banica\thanks{valeria.banica@sorbonne-universite.fr} \and Georg Maierhofer\thanks{Corresponding author; georg.maierhofer@sorbonne-universite.fr} \and Katharina Schratz\thanks{katharina.schratz@sorbonne-universite.fr}}
\date{Laboratoire Jacques-Louis Lions (UMR 7598), Sorbonne Universit\'e, UPMC,\\ 4 place Jussieu, 75005\\\ \vspace{0.2cm}\\\today}
\newcommand\corr[1]{\textcolor{black}{#1}}

\newcommand{\dd}{\mathrm{d}}

\renewcommand{\Re}{\mathrm{Re}\,}
\renewcommand{\Im}{\mathrm{Im}\,}

\makeatletter
\newcommand{\doublewidetilde}[1]{{%
		\mathpalette\double@widetilde{#1}%
}}
\newcommand{\double@widetilde}[2]{%
	\sbox\z@{$\m@th#1\widetilde{#2}$}%
	\ht\z@=.9\ht\z@
	\widetilde{\box\z@}%
}
\makeatother

\usepackage{soul} 

\date{}
\begin{document}

\pagenumbering{arabic}
\maketitle
\vspace{-0.75cm}
\begin{abstract}
We introduce a numerical approach to computing the Schr\"odinger map (SM) based on the Hasimoto transform which relates the SM flow to a cubic nonlinear Schr\"odinger (NLS) equation. In exploiting this nonlinear transform we are able to introduce the first fully explicit unconditionally stable symmetric integrators for the SM equation. Our approach consists of two parts: an integration of the NLS equation followed by the numerical evaluation of the Hasimoto transform. Motivated by the desire to study rough solutions to the SM equation, we also introduce a new symmetric low-regularity integrator for the NLS equation. This is combined with our novel fast low-regularity Hasimoto (FLowRH) transform, based on a tailored analysis of the resonance structures in the Magnus expansion and a fast realisation based on block-Toeplitz partitions, to yield an efficient low-regularity integrator for the SM equation. This scheme in particular allows us to obtain approximations to the SM in a more general regime (i.e. under lower regularity assumptions) than previously proposed methods. The favorable properties of our methods are exhibited both in theoretical convergence analysis and in numerical experiments.
\end{abstract}

\tableofcontents

\section{Introduction}
The Sch\"odinger map (SM) equation plays an important role as a model of a wide range of physical phenomena describing  inter alia {the dynamics of the orientation of the magnetization (or spin) in ferromagnetic materials in the isotropic case (as a special case of the Landau--Lifschitz equation \cite{landau1935theory,de2020landau}) and the evolution of the tangent vector field to a vortex filament in an ideal fluid (as the tangent dynamics to the vortex filament equations \cite{buttke1988numerical}).} The 1-D SM equation (for a closed curve) with values in $\mathbb{S}^2$ describes the evolution in time of a vector field $\mathbf{T}(t):\mathbb{T}\rightarrow \mathbb{S}^2$ under the following flow\ \vspace{-0.5cm}\\
\begin{align}\label{eqn:SMEqn}
	\begin{cases}\mathbf{T}_t=\mathbf{T}\wedge \mathbf{T}_{xx},\\
		\mathbf{T}(0,x)=\mathbf{T}_0(x),
	\end{cases}
\end{align}\ \vspace{-0.3cm}\\
{where $\mathbb{T}=\mathbb{R}/(2\pi\mathbb{Z})$ is the torus, i.e. we work with periodic boundary conditions in \eqref{eqn:SMEqn}.} We seek to solve numerically the Cauchy problem for this equation, which has been the subject of a significant amount of prior work. A first approach to this problem was developed by Buttke \cite{buttke1988numerical} introducing a structure preserving implicit finite difference-based scheme which allows the stable integration of the vortex filament equations (VFE) while preserving momentum and energy in the numerical flow. Due to the implicit structure of the scheme a restriction of type $\Delta t\lesssim (\Delta x)^2$ is thereby required to ensure convergence of the fixed point iterates for the solution of the nonlinear time-stepping equations. This was followed by later work by de la Hoz et al. \cite{de2009numerical,de2014vortex,de2022vortex} who used a fully explicit fourth-order Runge--Kutta (RK4) scheme for the SM equation which allows for  convergence at higher order under the stability constraint $\Delta t\lesssim (\Delta x)^2$. 
More recently several papers succeeded in developing unconditionally stable methods{. Amongst} them is work by {E} \& Wang \cite{weinan2001numerical} who developed a semi-implicit Gauss--Seidel projection method which is unconditionally stable for a general Landau--Lifschitz equation (which includes the case of the SM \eqref{eqn:SMEqn}). {Another efficient semi-implicit method} was proposed by \cite{xie2020second} based on the backward differentiation formula. However, as noted by \cite{xie2020second}, such semi-implicit schemes {based on linear multistep methods (in particular the backward differentiation formula) can only be constructed up to second order and for the construction of higher order schemes more sophisticated techniques would be required. This is because there are no linear multistep methods of order greater than two which are A-stable.} Finally, we note that in \cite{de2009numerical} it was also shown that an efficient numerical scheme can be constructed based on coupling a  RK4 method  with the stereographic projection which maps the SM equation to a nonlinear Schr\"odinger equation with rational nonlinearity. The latter approach allows for a weaker CFL condition $\Delta t=\mathcal{O}(\Delta x)$ for stability. 

In this manuscript we take a different route: We  exploit the so-called Hasimoto transform \cite{hasimoto1972soliton} which allows us to relate the SM to a \emph{cubic} NLS equation. {The advantage of the latter lies in the structure of the nonlinearity allowing for improved numerical treatment \cite{besse2002order,lubich2008splitting,thalhammer2008high,jin_markowich_sparber_2011,faou2012geometric,koch2013error} and, in particular, efficient numerical approximations at low regularity \cite{ostermann2018low,bruned2022resonance,maierhoferschratz23,fengmaierhoferschratz23,bronsard2022error}.}  In recent years this Hasimoto transform has been successfully applied to the theoretical study of rough solutions to the vortex filament equations in a sequence of papers by Banica \& Vega \cite{banica2009stability,banica2013,banica2020evolution} and de la Hoz \& Vega \cite{de2014vortex}, but so far it has not been exploited for computations. In this manuscript we aim to close this gap, with a particular interest set on low-regularity approximations to the SM which play an important role for example in the simulation of domain-wall states in spin chains (cf. \cite{gamayun2019domain,misguich2019domain}).
In order to achieve this, we introduce a new symmetric low-regularity integrator for the cubic NLSE as well as a  novel fast low-regularity Hasimoto (FLowRH) transform, based on a tailored analysis of the resonance structures in the Magnus expansion and a fast realisation based on block-Toeplitz partitions. {This new method allows us to obtain approximations to the SM equation under weaker regularity assumptions {than previously proposed methods.}

{Indeed as we will see below this new integrator is the first one to be able to provide guaranteed convergence for solutions with data $T_0\in H^{5/2+}$} and as such is the first available numerical scheme for the SM equation which can treat solutions of such low regularity. It turns out that our method also performs very well for solutions with much lower regularity (cf. the example given in section \ref{sec:application_to_VFE}). This is important because so far there was a large gap between theoretical existence results of solutions and numerical algorithms which can compute these solutions in practice. {The first rigorous studies of solutions to \eqref{eqn:SMEqn} were provided by Zhou and Guo \cite{ZhouGuo1984} and Sulem et al. \cite{sulem1986continuous}. They proved existence of weak solutions in $H^1$ (actually these results were stated for \eqref{eqn:SMEqn} posed on $\mathbb{R}$ but they can be extended to periodic boundary conditions, i.e. \eqref{eqn:SMEqn} posed on $\mathbb{T}$, see section~1.2 in \cite{jerrard_smets_2012}). Amongst further notable contributions on the subject is work by Nishiyama \& Tani \cite{nishiyama1994solvability,nishiyama1996initial,tani1997solvability} who proved existence and uniqueness of solutions for initial conditions in $H^4$. The latest global existence and uniqueness results are valid for solutions in $L^{\infty}(H^k),k\geq 2$ and were proved in a sequence of papers by Ding and Wang \cite{dingwang2001} (see also Chang et al. \cite{chang2000k} and Nahmod et al. \cite{nahmod2007schrodinger}, and Rodnianski et al. \cite{rodnianski2009global} for a more recent self-contained proof). A different approach to uniqueness in $H^2$ was provided by McGahagan \cite{mcgahagan2007approximation,mcgahagan2004some}. Further details about prior theoretical studies of solutions to \eqref{eqn:SMEqn} can be found in to \cite{jerrard_smets_2012} and the references therein.} {Finally, for initial conditions at least in $H^2$ the flow map of the SME is continuous in $H^1$ \cite[p.~642]{jerrard_smets_2012}.}

{The present numerical scheme described in section~\ref{sec:low_regularity_integrator} thus operates right at the edge of current existence results and provides simulations in practice {for regularity assumptions close to the limit of $H^1$ data needed in current existence results} (cf. section~\ref{sec:application_to_VFE}), thus improving significantly on previous {numerical methods}.}

{Additionally,} based on the Hasimoto transform we introduce in the smooth setting a wide class of unconditionally stable explicit symmetric integrators which {can be designed for any desired order of convergence} under sufficient regularity assumptions. {This means that we overcome any CFL-type condition (cf. \cite{buttke1988numerical,de2009numerical}) and semi-implicit nature and order restrictions (cf. \cite{weinan2001numerical,xie2020second}) of prior work} {while at the same time allowing for better preservation of the underlying geometric structure of the equation.}

{Our theoretical results are confirmed and underlined in detailed numerical experiments in section~\ref{sec:numerical_experiments} where we observe that our fully explicit methods provide a significant performance improvement over prior work (in particular the state-of-the-art by Xie et al. \cite{xie2020second}) and that our low-regularity integrator for the SM equation is able to perform very well even in regimes where classical methods are observed to fail (rougher initial conditions). Both new schemes are also observed to lead to good energy preservation properties.}

\vspace{-0.2cm}
\section{The Hasimoto transform}\label{sec:hasimoto_transform}
\vspace{-0.1cm}
Our novel approach for the numerical study of the solution of \eqref{eqn:SMEqn} is to exploit the Hasimoto transform \cite{hasimoto1972soliton}. This transform has been successfully exploited in analytical studies of solutions to \eqref{eqn:SMEqn} (see for example \cite{hasimoto1972soliton,Calini_2011,banica2009stability,de2014vortex,banica2013,banica2020evolution,LAFORTUNE2013766}) but, to the best of our knowledge, this has not been used in numerical analysis literature yet. {Hasimoto \cite{hasimoto1972soliton} showed that the solution $\mathbf{T}(t): \mathbb{T}\rightarrow \mathbb{S}^2$ of \eqref{eqn:SMEqn} can be found via the following process from a given initial condition $\mathbf{T}_0$.} {In the interest of clarity we present the Hasimoto transform here using the Frenet frame, but we note that it is possible to use instead the parallel frame described below from a given $\mathbf{T}_0$ (cf. \cite[Theorem~1.3]{koiso97} and also \cite{koiso1996vortex,tani1997solvability,banica2013stability}) thus avoiding issues related to vanishing curvature arising in the Frenet-based construction. We denote by} $\kappa,\tau$ the `curvature' and `torsion' corresponding to $\mathbf{T}$ (this terminology arises from the interpretation of $\mathbf{T}$ as the tangent vector field to a closed vortex filament, cf. \cite{hasimoto1972soliton} and see also the example in section \ref{sec:application_to_VFE}). These satisfy
\begin{align}\label{eqn:computation_curvature_torsion}
	\kappa(t,x)=\left\|\mathbf{T}_x(t,x)\right\|_{\ell^2},\quad
	\kappa^2\tau(t,x)=-\langle\mathbf{T}_{xx}(t,x)\times \mathbf{T}(t,x), \mathbf{T}_x(t,x)\rangle_{\ell^2},
\end{align}
where $\|\cdot\|_{\ell^2}$ denotes the standard Euclidean norm on $\mathbb{R}^3$. {Suppose for the time being that the curvature does not vanish ($\kappa\neq 0$).} Firstly, we compute $\kappa_0$ and $\tau_0$ {(the initial values for $\kappa$ and $\tau$)} from \eqref{eqn:computation_curvature_torsion} and let $\mathbf{N}_0:=\kappa_0^{-1}\partial_x\mathbf{T}_0$, $\mathbf{B}_0:=\mathbf{T}_0\times \mathbf{N}_0$. Then we define an initial orthonormal frame $\mathbf{T}_0,\mathbf{e}_{1,0},\mathbf{e}_{2,0}:\mathbb{T}\rightarrow \mathbb{S}^2$ by taking
\begin{align*}
    \mathbf{e}_{1,0}(x)&:=\cos(\tau_0(x))\mathbf{N}_0(x)-\sin(\tau_0(x))\mathbf{B}_0(x),\quad   \mathbf{e}_{2,0}(x):=\sin(\tau_0(x))\mathbf{N}_0(x)+\cos(\tau_0(x))\mathbf{B}_0(x).
\end{align*}
Finally, the value of $\mathbf{T}(t,x)$ is found by solving, for $t\in [0,T]$,
\begin{align}\label{eqn:temporal_hasimoto_transform}
	\frac{\partial}{\partial t}\begin{pmatrix}
\mathbf{T}\\
\mathbf{e}_1\\
\mathbf{e}_2
\end{pmatrix}(t,x)=\mathsf{A}(t,x)\begin{pmatrix}
\mathbf{T}\\
\mathbf{e}_1\\
\mathbf{e}_2
\end{pmatrix}(t,x),\,\text{where\ }\mathsf{A}(t,x)=\begin{pmatrix}
	0 & -\Im u_x & \Re u_x \\ \Im u_x & 0 & -\frac{|u |^2}{2} \\ -\Re u_x &  \frac{|u |^2}{2}& 0
\end{pmatrix}(t,x),
\end{align}
and where $u$ is the solution to the cubic nonlinear Schr\"odinger equation
\begin{align}\label{eqn:NLS_in_hasimoto_transform}
\begin{cases}	iu _t+u _{xx}+\frac 1 2|u |^2u =0,& (t,x)\in\mathbb{R}^+\times \mathbb{T},\\
	u(0,x)=\kappa_0(x)\exp(i\int_0^x \tau_0(\tilde{x})\dd \tilde{x}),
	\end{cases}
\end{align}
{where, as above, $\mathbb{T}=\mathbb{R}/(2\pi\mathbb{Z})$ is the torus meaning we consider periodic boundary conditions. Note, if $\mathbf{T}_0$ is not the tangent vector field of a closed curve we have to include an additional factor in the initial condition for $u$ in order to ensure its periodicity, $u(0,x)=\kappa_0(x)\exp(i\int_{0}^x\tau_0(\tilde{x})\dd\tilde{x}-i\int_0^{2\pi}\tau_0(\tilde{x})\dd\tilde{x})$.} If we think of $\mathbf{T},\mathbf{e}_1,\mathbf{e}_2$ as row-vectors then \eqref{eqn:temporal_hasimoto_transform} can also be written as an evolution of a $3\times 3$ matrix:
\begin{align}\label{eqn:3x3_matrix_formulation_temporal_hasimoto}
\mathsf{y}_t(t,x)=\mathsf{A}(t,x)\mathsf{y}(t,x),\quad 
	\mathsf{y}(t,x):=\begin{pmatrix}
		\mathbf{T}\\
		\mathbf{e}_1\\
		\mathbf{e}_2
	\end{pmatrix}(t,x)
\end{align}
{The frame $\{\mathbf{T},\mathbf{e}_{1},\mathbf{e}_2\}$ is called the `parallel frame' because one can compute that the variations $\partial_x\mathbf{e}_{1}, \partial_x\mathbf{e}_{2}$ are only in the $\mathbf{T}$-direction. This frame can be constructed even when the curvature vanishes (cf. \cite[Theorem 1.3]{koiso97}). For instance, in the case when $\mathbf{T}_0$ is flat (i.e. there is a fixed $\mathbf{b}\in\mathbb{S}^2$ such that $\mathbf{T}_0(x)\cdot\mathbf{b}=0$ for all $x\in\mathbb{T}$) we can define $\tau_0:=0$, take $\mathbf{B}_0:=\mathbf{b}$ and $\mathbf{N}_0:=\mathbf{B}_0\times\mathbf{T}_0$ and proceed as above to constructing the initial parallel frame $\mathbf{e}_{1,0},\mathbf{e}_{2,0}$ before integrating the temporal evolution \eqref{eqn:3x3_matrix_formulation_temporal_hasimoto}. }

\subsection{Properties of the SM equation}\label{sec:properties_SM_equation}
We note that the Hasimoto transform provides a useful tool for understanding the theoretical properties of the SM equation. Firstly, it allows us to prove existence results such as the following: if $\mathbf{T}_0$ is smooth then we have existence and uniqueness of smooth global solutions to \eqref{eqn:SMEqn}. This result is a standard consequence of existence of solutions to NLS and Cauchy--Lipschitz applied to \eqref{eqn:temporal_hasimoto_transform}. As recalled in the introduction the Cauchy well-posedness theory for SM was intensively studied over the recent decades.

Secondly, the Hasimoto transform immediately allows us to relate conserved quantities from the NLS flow directly to conservation laws for the SM equation. One example is the conservation of the mass $\int_{\mathbb{T}}|u|^2\dd x$ in the NLS, which translates to conservation of energy $\mathcal{E}$ of the SM. {Two important conserved quantities of the SM equation} (cf. \cite[Lemma~A.1]{jerrard_smets_2012}) are given by
\begin{align}\label{eqn:definition_conserved_quantities_SM}
	\mathcal{E}(t)=\int_{\mathbb{T}}\|\mathbf{T}_x(t,x)\|_{\ell^2}^2\dd x,\quad \mathcal{I}(t)=\int_{\mathbb{T}}\|\mathbf{T}_t(t,x)\|_{\ell^2}^2+\|\mathbf{T}_{xx}(t,x)\|_{\ell^2}^2{-\frac{3}{4}}\|\mathbf{T}_x(t,x)\|_{\ell^2}^4\dd x.
\end{align}
{The preservation of these two constants of motion is particularly important since they control the first two spatial derivatives of the solution $\mathbf{T}$ (cf. \cite[Appendix~A]{jerrard_smets_2012}), meaning that good preservation of $\mathcal{E},\mathcal{I}$ in the numerical flow helps to avoid spurious blow-up arising from numerical instability in simulations.} {We will use these conserved quantities as a benchmark for the structure preservation properties of our symmetric numerical schemes (cf. section~\ref{sec:numerical_experiments}).}

\section{The smooth setting: Explicit unconditionally stable {SM integrators}}\label{sec:CFL_free_high_regularity}

In this section we will introduce a large class of unconditionally stable fully explicit symmetric numerical methods for the SM equation in the \emph{smooth setting}. Through \eqref{eqn:temporal_hasimoto_transform} the bulk of the challenges in the nonlinearity is pushed into the NLS part of the equation. This system is well-studied and a number of efficient techniques exist for its solution {(cf. \cite{hochbruck_ostermann_2010,mclachlan_quispel_2002,jin_markowich_sparber_2011})}. In the following we will particularly focus on the integration of the temporal Hasimoto transform \eqref{eqn:temporal_hasimoto_transform} meaning we are given the complete initial data $\mathbf{T}(0,x),\mathbf{e}_1(0,x),\mathbf{e}_2(0,x), u(0,x), x\in\mathbb{T}.$ Motivated by \cite[Chapter~XI]{hairer2013geometric} \& \cite{faou2004energy} we seek to construct symmetric numerical schemes, which promise not only stability but also good long-time preservation of actions of the SM, including the energy $\mathcal{E}$ and $\mathcal{I}$ as defined in \eqref{eqn:definition_conserved_quantities_SM}. A symmetric method preserves the time-reversibility of the exact flow of the equation. In particular if we denote by $\Phi_{m\to m+1}: \mathbf{T}_{m}\mapsto\mathbf{T}_{m+1}$ the (nonlinear) map corresponding to our time-discretisation, we have the following standard definition:

\begin{definition}[{See for example \cite{hairer2013geometric}}]\label{def:time-symmetric_method} The method $\Phi_{m\to m+1}$ is called \textit{symmetric} if $\Phi_{m\to m+1}=\Phi_{m+1\to m}^{-1}$.
\end{definition}

Throughout the present work we consider a spectral collocation method as the spatial discretisation and our main focus will be to develop the semi-discretisation in time.
\subsection{Symmetric splitting methods for the NLS}\label{sec:splitting_for_nls}
The first step in constructing such a symmetric method is to decide on an integrator for the NLS equation. Splitting methods \cite{mclachlan_quispel_2002} are a good choice because for the cubic NLS we can split the nonlinear part from the dispersive part in the system and efficiently integrate both in turns. In particular letting\\
\begin{minipage}{0.495\textwidth}\centering\begin{align*}
	(P1)\begin{cases}
		u_t(t,x)=iu_{xx}(t,x),\\
		u(0,x)=u_1(x),
	\end{cases}
\end{align*}
\end{minipage}
\begin{minipage}{0.495\textwidth}\centering
\begin{align*}
	(P2)\begin{cases}u_t(t,x)=\frac{i}{2}|u(t,x)|^2u(t,x),\\
	u(0,x)=u_2(x).
	\end{cases}
\end{align*}\end{minipage}
we know that the solution to (P1) is given by $u(t,x):=\varphi^{(P1)}_t(u_1(x))=\exp(it\partial_x^2)u_1(x)$, which can be computed efficiently for a spectral discretisation in $\mathcal{O}(N\log N)$ operations using the FFT. Moreover the exact solution to (P2) is given by $u(t,x):=\varphi^{(P2)}_t(u_2(x))=\exp(it|u_2(x)|^2/2)u_2(x),$ which again is highly efficient because it requires only diagonal operations on the discrete grid of function values i.e.\ only $\mathcal{O}(N)$ operations. If we let $D_\Delta$ and $D_N$ be the Lie derivatives (cf. for instance \cite[Appendix~A.1]{koch2013error}) of $\varphi^{(P1)}_t,\varphi^{(P2)}_t$ respectively, a general splitting method takes the form
\begin{align}\label{eqn:general_form_splitting_method}
	u_{m+1}=e^{a_1h D_\Delta}e^{b_1h D_{N}}\cdots e^{a_Sh D_{\Delta}}e^{b_Sh D_{N}}u_{m},
\end{align}
{where $u_m,u_{m+1}$ correspond to the approximations of $u(t_m),u(t_{m+1})$ at given times $t_m,t_{m+1}\in\mathbb{R}_{\geq 0}$ with $t_{m+1}>t_m$, $h=t_{m+1}-t_m$ is the time step (whose potential dependency on $m$ we suppress in the above formula for ease of notation), and where} $a_s, b_s\in \mathbb{R},\, 1\leq s\leq S, S\in\mathbb{N},$ are constants that do not depend on $m$. We shall see in Thm.~\ref{thm:convergence_thm_splitting_NLS} below that for any given $p\in\mathbb{N}$ we can find algebraic conditions on $a_s,b_s, 1\leq s\leq S,$ which ensure that the method is convergent at order $p$, directly in the full unbounded operator setting. This means that splitting methods are unconditionally stable and convergent. Moreover, the form \eqref{eqn:general_form_splitting_method} facilitates an obvious characterization of symmetric splitting methods (cf. \cite{hairer2013geometric,blanes2013optimized}):
\begin{lemma}
	Suppose one of $a_1,b_1$ and one of $a_S,b_S$ is nonzero. Then the method defined by \eqref{eqn:general_form_splitting_method} is symmetric if and only if one of the following two cases holds
	\begin{itemize}
		\item $a_1\neq 0$, $b_S=0$ and $a_{S-s}=a_s, b_{S-1-s}=b_{s}$ for all $1\leq s\leq S-1$;
		\item $b_1\neq 0$, $a_S=0$ and $b_{S-s}=b_s, a_{S-1-s}=a_{s}$ for all $1\leq s\leq S-1$.
	\end{itemize}
\end{lemma}

\begin{example}[Strang splitting, cf. \cite{strang1968construction}] An obvious example is the Strang splitting for the nonlinear Schr\"odinger equation \eqref{eqn:NLS_in_hasimoto_transform} which takes one of two forms
	\begin{align*}
		u_{m+1}=\varphi^{(P1)}_{\frac{h}{2}}\circ\varphi^{(P2)}_{h}\circ\varphi^{(P1)}_{\frac{h}{2}}(u_m),\quad\text{or\quad}u_{m+1}=\varphi^{(P2)}_{\frac{h}{2}}\circ\varphi^{(P1)}_{h}\circ\varphi^{(P2)}_{\frac{h}{2}}(u_m).
	\end{align*}
\end{example}

\subsubsection*{Recall of previous convergence results}
We quickly recap a central result which will be useful for our analysis in section~\ref{sec:convergence_analysis_cfl_high_regularity_integrators}. The result was first proved for the cubic nonlinear Schr\"odinger equation by \cite{lubich2008splitting} for Strang splitting, and then stated by \cite{koch2013error} for splitting methods of arbitrary order. To begin with we note that replacing $D_{\Delta},D_{N}$ with bounded operators one can, using Taylor series expansions, easily find algebraic conditions which ensure that the splitting method \eqref{eqn:general_form_splitting_method} is convergent of order $p$. We call this the \textit{non-stiff order} of the splitting method. Then it turns out that for the nonlinear Schr\"odinger equation this order is preserved also in the stiff case:

\begin{theorem}[{Thm.~3.5 in \cite{koch2013error} extending on Thm.~7.1 from \cite{lubich2008splitting}}]\label{thm:convergence_thm_splitting_NLS} If $u:[0,T]\times \mathbb{T}\rightarrow \mathbb{C}$ is the solution to \eqref{eqn:NLS_in_hasimoto_transform}, {with $u\in\mathcal{C}^{0}(0,T;H^{2p+1}(\mathbb{T}))$,} then for any exponential operator splitting method \eqref{eqn:general_form_splitting_method} with \textit{non-stiff order} $p$ {and uniform time steps, $t_{m+1}-t_{m}=h\,\, \forall 0\leq m\leq \lfloor T/h\rfloor,$} we have that 
\begin{align*}
	\left\|u_m-u(t_m)\right\|_{H^1}\leq Ch^{p}, \quad \forall\, t_m\leq T,
\end{align*}
where the constant $C>0$ depends on $T$ and $\sup_{t\in[0,T]}\|u(t)\|_{H^{2p+1}}$.
\end{theorem}

In an analogous manner a similar convergence result can be shown for quasi-uniform time steps. In the following statement we will use the standard expression $\rho$-quasiuniform, $0<\rho<1$, for a sequence $(a_m)_{1\leq m\leq M}$ to mean that $\min_{1\leq m\leq M} |a_m|>\rho \max_{1\leq m\leq M} |a_m|$.

\begin{cor}\label{cor:convergence_thm_splitting_NLS_variable_time_step} Let $u$ be as above{, i.e. it is the unique solution $u\in\mathcal{C}^0(0,T;H^{2p+1}(\mathbb{T}))$ to \eqref{eqn:NLS_in_hasimoto_transform},} and consider the exponential operator splitting method \eqref{eqn:general_form_splitting_method} with \textit{non-stiff order} $p$ but applied with variable {step sizes} $h_m=t_{m+1}-t_m>0$ such that $T=\sum_{m=1}^M h_m$. Let $\rho\in(0,1)$, then there is a constant $C$ depending on $T, \sup_{t\in\mathbb{T}}\|u\|_{H^{{2p+1}}}$ and $\rho$, such that for every $\rho$-quasiuniform sequence of time steps we have
	\begin{align*}
		\left\|u_m-u(t_m)\right\|_{{H^1}}\leq C\left(\max_{1\leq m\leq M}h_m\right)^p, \quad \forall\, 1\leq m\leq M.
	\end{align*}
\end{cor}
\subsection{Time integration of the Hasimoto transform}\label{sec:time_integration_hasimoto_transform}
Having solved the NLS \eqref{eqn:NLS_in_hasimoto_transform} numerically we turn to the time integration of the linear matrix ODE \eqref{eqn:3x3_matrix_formulation_temporal_hasimoto}. If $u\in\mathcal{C}([0,T];H^{1})$ then, for almost every value of $x\in\mathbb{T}$, the matrix function $t\mapsto \mathsf{A}(t,x)$ is uniformly bounded on compacta thus we can express (see \cite[\S~4.1]{iserles_munthe-kaas_norsett_zanna_2000} and \cite{magnus1954exponential}) the solution to this system as $\mathsf{y}(t,x)=\exp\left(\mathsf{B}(t,x)\right)\mathsf{y}(0,x)$ where $\mathsf{B}: \mathbb{R}_+\times \mathbb{T}\rightarrow \mathbb{R}^{3\times 3}$ is given by its Magnus expansion $\mathsf{B}(t,x)=\sum_{k=1}^\infty \mathsf{H}_k(t,x)$ (which holds true for almost every $x$ pointwise), with
\begin{align*}
\mathsf{H}_1(t,x)&=\int_0^t\mathsf{A}(\xi_1)\dd \xi_1,\quad	\mathsf{H}_2(t,x)=\frac{1}{2}\int_0^t\int_0^{\xi_1}\left[\mathsf{A}(\xi_1),\mathsf{A}(\xi_2)\right]\dd\xi_2\dd \xi_1,\\
	\mathsf{H}_3(t,x)&=\frac{1}{6}\int_0^h\int_{0}^{\xi_1}\int_0^{\xi_2}\left(\left[\mathsf{A}(\xi_1),\left[\mathsf{A}(\xi_2),\mathsf{A}(\xi_3)\right]\right]+\left[\mathsf{A}(\xi_3),\left[\mathsf{A}(\xi_2),\mathsf{A}(\xi_1)\right]\right]\right)\dd\xi_3\dd\xi_2\dd \xi_1,
\end{align*}
{where the square brackets $[\,\cdot\,,\,\cdot\,]$ denote the commutator of the two matrix arguments, e.g. $[\mathsf{A}(\xi_1),\mathsf{A}(\xi_2)]=\mathsf{A}(\xi_1)\mathsf{A}(\xi_2)-\mathsf{A}(\xi_2)\mathsf{A}(\xi_1)$.} The terms $\mathsf{H}_{j}, j\geq 3,$ take a similar form: they can be expressed in terms of iterated integrals of commutators of $\mathsf{A}$. A general expression of these terms can be found using rooted trees, see for example \cite[\S~4.1]{iserles_munthe-kaas_norsett_zanna_2000}. We shall denote by $\mathsf{B}^{(K)}:=\sum_{k=0}^{K}\mathsf{H}_k$ the truncated Magnus series. Truncating the Magnus expansion preserves the symmetry of the flow:

\begin{theorem}[{See Thm.~3 in \cite{ISERLES2001379}}] The truncated Magnus expansion, i.e. the numerical method $\Phi_{h}^{(K)}: \mathsf{y}_m\mapsto \mathsf{y}_{m+1}=\exp(\mathsf{B}^{(K)}(t_m,t_{m+1},x))\mathsf{y}_m$, is time-symmetric in the sense of definition \ref{def:time-symmetric_method}.
\end{theorem}

Of course, we can solve \eqref{eqn:NLS_in_hasimoto_transform} only on a discrete temporal grid, thus we do not have access to $u(t,x)$ for all values of $t\in[0,T]$. Therefore we have to use a quadrature rule on the integral expressions for $\mathsf{H}_k(t,x)$ which relies on solution values of \eqref{eqn:NLS_in_hasimoto_transform} only at discrete time steps. This results in a so-called \textit{interpolatory Magnus integrator} (introduced by \cite{iserles1999solution,iserles_munthe-kaas_norsett_zanna_2000}) and can be constructed in the following way: We follow the notation of \cite{iserles1999solution} and note that we can write every $\mathsf{H}_{k}$ as a linear combination of integrals:
\begin{align*}
I_h=\int_{0}^h\int_{0}^{\xi_{i_1}}\cdots\int_{0}^{\xi_{i_k}}\mathcal{L}(\mathsf{A}(\xi_1),\mathsf{A}(\xi_2),\dots,\mathsf{A}(\xi_{k}))\dd \xi_{k}\dd \xi_{k-1}\cdots\dd\xi_{1},
\end{align*}
where $\mathcal{L}$ is a multilinear map (consisting of $k-1$ nested commutators) and $i_j\in{1,\dots,k}$ for all $1\leq j\leq k$. We take Gauss--Legendre quadrature with $L=\lceil(p+1)/2\rceil$ nodes, $0\leq c_l\leq 1$, and let $p_{l}(t)=\prod_{\substack{i=1}{i\neq l}}^{L}(t-c_i)/(c_{l}-c_{i}),\, 1\leq l\leq L,$ be the corresponding Lagrange interpolation polynomials. The interpolatory Magnus integrator is then found by precomputing the quadrature weights
\begin{align}\label{eqn:definition_of_b_l_quadrature_weights}
	b_{\mathbf{l}}=\int_{0}^h\int_{0}^{\xi_{i_1}}\cdots\int_{0}^{\xi_{i_k}}\prod_{i=1}^{k} p_{l_i}(\xi_i/h)\dd \xi_{k}\dd \xi_{k}\cdots\dd\xi_{1}=h^{k}\int_{0}^1\int_{0}^{\xi_{i_1}}\cdots\int_{0}^{\xi_{i_k}}\prod_{i=1}^{k} p_{l_i}(\xi_i)\dd \xi_{k}\dd \xi_{k}\cdots\dd\xi_{1},
\end{align}
and then replacing the integrals $I_h$ in the truncated Magnus expansion by the quadrature approximations
\begin{align}\label{eqn:multivariate_quadrature_extension}
	\mathcal{Q}[I_h]=\sum_{\mathbf{l}}b_{\mathbf{l}}\mathcal{L}\left(\mathsf{A}(hc_{l_1}),\mathsf{A}(hc_{l_2}),\dots, \mathsf{A}(hc_{l_L})\right).
\end{align}

As noted in \cite{iserles1999solution} these approximations are a good choice for two reasons: firstly, the resulting \textit{interpolatory Magnus integrator} is time-symmetric due to the symmetric distribution of Gaussian quadrature points in $[0,1]$ (cf. \cite[p.~392]{ISERLES2001379}). Secondly, the resulting numerical integrator for \eqref{eqn:temporal_hasimoto_transform} is of order $p$ as we shall see in the results quoted below. We denote the resulting approximation to $\mathsf{H}_k(h,u)$ by $\mathcal{Q}_{h}\left[\mathsf{H}_k;u\right](x)$ to emphasize that we use function values of $u,\partial_x u$ at $hc_l,l=1,\dots L,$ in the evaluation of $\mathsf{A}$ of this process.

\subsubsection*{Summary of previous convergence results}

The error analysis of interpolatory Magnus integrators is very well understood and was first given in \cite{iserles1999solution}. The following central results concern the local error resulting from truncation of the Magnus expansion and then the error committed in applying the Gaussian quadrature.

\begin{theorem}[{Rephrased from Thm.~2.7 in \cite{iserles1999solution}}]\label{thm:truncated_magnus_error}
	The truncated Magnus series satisfies
	\begin{align*}
		\left\|\mathsf{B}(h,x)-\sum_{k=1}^p \mathsf{H}_k(h,x)\right\|_{\ell^2}&\leq Ch^{p+1}\left(\sup_{t\in[0,h]}\left\|\mathsf{A}(t,x)\right\|_{\ell^2}\right)^{p+1},
	\end{align*}
where the norm $\|\cdot\|_{\ell^2}$ is the usual Euclidean matrix norm and $C>0$ is a constant independent of $\mathsf{A},h$.
\end{theorem}

\begin{theorem}[{Cor.~3.3 in \cite{iserles1999solution}}]\label{thm:interpolatory_magnus_quadrature_error}
	If $L=\lceil(p+1)/2\rceil$, the quadrature introduced in \eqref{eqn:multivariate_quadrature_extension} is at least of order $p$ for any integral $I_h$ appearing in the Magnus expansion, in particular there is a constant $C$ depending on the number of iterated integrals in $I_h$ but independent of $h$ and $\mathsf{A}$ such that
	\begin{align*}
		\left\|I_h-\mathcal{Q}[I_h]\right\|_{\ell^2}\leq C h^{p+1}\max_{l_1+\cdots l_L=p+1}\sup_{\substack{\xi_{j}\in[0,h]\\j=1,\dots, N}}\left\|\mathcal{L}\left(\partial_t^{l_1}\mathsf{A}(\xi_1),\dots,\partial_t^{l_L}\mathsf{A}(\xi_L)\right)\right\|_{\ell^2}.
	\end{align*}
\end{theorem}

\subsection{Novel integrators and convergence analysis for regular solutions}\label{sec:convergence_analysis_cfl_high_regularity_integrators}
Having understood the building blocks of our novel method for the SM equation \eqref{eqn:SMEqn} we briefly summarize the combined method before providing a detailed convergence analysis.

\paragraph{Scheme A for the approximation of $\mathbf{T}$:}

We use the notation and parameters introduced in the previous two sections. We fix $p$ the desired order of the method, $h$ the coarse time step, and $T$ the end time and proceed as follows:

\begin{enumerate}
	\item From the given initial data $u_0(x), x\in\mathbb{T},$ we solve the NLS \eqref{eqn:NLS_in_hasimoto_transform} using a splitting method of order $p$ as described in section~\ref{sec:splitting_for_nls} at the $LM$-time steps given by $t_{m,l}=h(m-1)+hc_l$, where $M=T/h$ and $L=\lceil (p+1)/2\rceil$. We denote those approximate values by $\mathbf{u}^{(m)}_{approx}(x)=(u_{m,1}(x),\dots,u_{m,L}(x))\approx (u(t_{m,1},x),\dots,u(t_{m,L},x))$.
	\item From the given initial frame $\mathbf{T}_0(x),\mathbf{e}_{1,0}(x),\mathbf{e}_{2,0}(x),\,x\in\mathbb{T},$ we apply an interpolatory Magnus integrator as described in section~\ref{sec:time_integration_hasimoto_transform} to \eqref{eqn:temporal_hasimoto_transform} in order to obtain the approximation of $\mathbf{T}(t_m,x)$ at new time values $t_m=mh, 1\leq m\leq M$. In particular we define the numerical map $\Phi_{m\to m+1}$ by
	\begin{align*}
		\mathsf{y}_{m+1}(x)=\Phi_{m\to m+1}(\mathsf{y}_m)=\exp\left(\sum_{k=1}^p\mathcal{Q}_h\left[\mathsf{H}_k;\mathbf{u}_{approx}^{(m)}\right](x)\right)\mathsf{y}_m(x).
	\end{align*}
\end{enumerate}

\subsubsection*{Convergence analysis of the full method}

Firstly, we note the following simple stability bound which is an immediate consequence of the skew-symmetry of $\mathsf{A}(t,x)$ for all $t\in[0,T],x\in\mathbb{T}$:
\begin{proposition}[Stability]\label{prop:stability_full_method}
	For any $\mathbf{v},\mathbf{w}:[0,T]\times \mathbb{T}\rightarrow \mathbb{R}^{3}$, $\mathbf{u}:\mathbb{T}\rightarrow \mathbb{C}^{L}$, we have
	\begin{align*}
	\left\|\exp\left(\sum_{k=1}^p\mathcal{Q}_h\left[\mathsf{H}_k;\mathbf{u}\right](x)\right)\mathbf{v}-\exp\left(\sum_{k=1}^p\mathcal{Q}_h\left[\mathsf{H}_k;\mathbf{u}\right](x)\right)\mathbf{w}\right\|_{\ell^2}=\left\|\mathbf{v}-\mathbf{w}\right\|_{\ell^2}.
	\end{align*}
\end{proposition}
\begin{proof}
 By definition (cf. \eqref{eqn:temporal_hasimoto_transform}) the matrix $\mathsf{A}$ is skew-symmetric, thus one can easily check that the same applies to the quadratures of the terms in the Magnus expansion as these are just linear combinations of pointwise evaluations of $\mathsf{A}$. Thus we have that
 $\mathsf{C}(x):=\sum_{k=1}^p\mathcal{Q}_h\left[\mathsf{H}_k;\mathbf{u}\right](x)$ is skew-symmetric. We conclude by noting that the exponential of a skew-symmetric matrix is an isometry on $\ell^2$:
\begin{center}
	$\left\|\exp(\mathsf{C})(\mathbf{v}-\mathbf{w})\right\|_{\ell^2}^2=\left\langle\mathbf{v}-\mathbf{w},\exp(\mathsf{C}^T)\exp(\mathsf{C})(\mathbf{v}-\mathbf{w})\right\rangle_{\ell^2}=\left\langle\mathbf{v}-\mathbf{w},\mathbf{v}-\mathbf{w}\right\rangle=\|\mathbf{v}-\mathbf{w}\|_{\ell^2}^2.$\vspace{-0.2cm}
\end{center}
\end{proof}
{\begin{remark}
    In particular this proposition exhibits a further advantage of the construction of integrators for \eqref{eqn:SMEqn} using the Hasimoto transform, namely that the geometric constraint $\|\mathbf{T}(t,x)\|_{\ell^2}=1$ is automatically conserved ensuring that the scheme indeed maps values from $\mathbb{S}^2$ back into $\mathbb{S}^2$. We note that in prior methods (cf. \cite{xie2020second}) this constraint is not automatically conserved and that in those cases typically each time step has to be augmented with an artificial projection step onto $\mathbb{S}^2$.
\end{remark}}

\begin{remark}
	Crucially, the numerical Hasimoto transform is unconditionally stable and the same is true for splitting methods of the NLS part, and so our overall method is indeed unconditionally stable.
\end{remark}

In addition we can combine Thms.~\ref{thm:convergence_thm_splitting_NLS}, \ref{thm:truncated_magnus_error} \& \ref{thm:interpolatory_magnus_quadrature_error} to find the local error of the method as follows: Let us denote by $\phi_{t_m\to t_{m+1}}:\mathbb{R}^{3\times 3}\rightarrow \mathbb{R}^{3\times3}$ the exact flow of the equation \eqref{eqn:3x3_matrix_formulation_temporal_hasimoto} so that the first row of this map corresponds to the exact solution of the SM equation \eqref{eqn:SMEqn}. Then we have the following error estimate:
\begin{theorem}[Local error]\label{thm:local_error_high_regularity}
Given $p\geq 1$, for any $\mathsf{y}_0:\mathbb{T}\rightarrow O(n;\mathbb{R})$ (i.e. a function taking values in the set of orthogonal matrices) and any $m=0,\dots M-1$ we have
	\begin{align*}
	\left\|\Phi_{m\to m+1}\left(\mathsf{y}_0\right)-\phi_{t_m\to t_{m+1}}\left(\mathsf{y}_0\right)\right\|_{L^{2}\left(\mathbb{T};\ell^2\right)}\leq C h^{p+1}
\end{align*}
where $C$ depends on $p$ and $\sup_{t\in[t_m,t_{m+1}]}\|u(t)\|_{H^{2p+1}}$.
\end{theorem}
\begin{proof}For notational simplicity we prove the statement for $m=0$ and note that the case $m\geq 1$ follows analogously. Note by section~\ref{sec:time_integration_hasimoto_transform} and our construction of the numerical method we have
	\begin{align*}
		\phi_{t_0\to t_{1}}\left(\mathsf{y}_0\right)(x)&=\exp(\mathsf{B}(h,x))\mathsf{y}_0(x),\quad\text{and}\quad
		\Phi_{0\to 1}\left(\mathsf{y}_0\right)(x)=\exp\left(\sum_{k=1}^p\mathcal{Q}_h\left[\mathsf{H}_k;\mathbf{u}^{(0)}_{approx}\right](x)\right)\mathsf{y}_0(x).
	\end{align*}
In the following let us denote by $\mathbf{u}:\mathbb{T}\rightarrow \mathbb{C}^L$ the vector of exact function values $\mathbf{u}(x)=(u(c_1h,x),\dots, u(c_Lh,x))$. We note that $\exp:\mathbb{R}^{3\times 3}\rightarrow\mathbb{R}^{3\times 3}$ is locally Lipschitz, and thus we have
	\begin{align}\nonumber
		\left\|\Phi_{m\to m+1}\left(\mathsf{y}_0\right)(x)-\phi_{t_m\to t_{m+1}}\left(\mathsf{y}_0\right)(x)\right\|_{\ell^2}&\leq C\left\|\mathsf{B}(h,x)-\sum_{k=1}^p\mathcal{Q}_h\left[\mathsf{H}_k;\mathbf{u}^{(0)}_{approx}\right](x)\right\|_{\ell^2}\underbrace{\|\mathsf{y}_0(x)\|_{\ell^2}}_{=1}\\\nonumber
		&\hspace{-2cm}\leq C\left(\left\|\mathsf{B}(h,x)-\sum_{k=1}^p\mathsf{H}_k(h,x)\right\|_{\ell^2}+\sum_{k=1}^p\left\|\mathsf{H}_k(x)-\mathcal{Q}_h\left[\mathsf{H}_k;\mathbf{u}\right](x)\right\|_{\ell^2}\right)\\\label{eqn:Lipschitz_estimate}
		&+C\sum_{k=1}^p\left\|\mathcal{Q}_h\left[\mathsf{H}_k;\mathbf{u}\right](x)-\mathcal{Q}_h\left[\mathsf{H}_k;\mathbf{u}^{(0)}_{approx}\right](x)\right\|_{\ell^2},
	\end{align}
where $C$ depends on $\sup_{t\in[t_m,t_{m+1}]}(\|u(t)\|_\infty+\|\partial_x u(t)\|_{\infty})+\max_{1\leq j\leq L}\left\|\left(\mathbf{u}^{(0)}_{approx}\right)_{j}\right\|_{\infty}+\left\|\left(\partial_x\mathbf{u}^{(0)}_{approx}\right)_{j}\right\|_{\infty}$, i.e. can be bounded above (using Morrey's inequality and 
{Cor.~\ref{cor:convergence_thm_splitting_NLS_variable_time_step}}) by a function of $\sup_{t\in[t_m,t_{m+1}]}\|u(t)\|_{H^{2p+1}}$. We can now use the estimates from Thm.~\ref{thm:truncated_magnus_error} to see that
\begin{align*}
	\left\|\mathsf{B}(h,x)-\sum_{k=1}^p\mathsf{H}_k(h,x)\right\|_{\ell^2}\leq C_1 h^{p+1}\left(\sup_{t\in[0,h]}\|\mathsf{A}(t,x)\|_{\ell^2}\right)^{p+1}
\end{align*}
for some constant $C_1>0$ independent of $u$. Thus
\begin{align}\label{eqn:local_error_truncated_magnus}
	\left\|\mathsf{B}(h,\cdot)-\sum_{k=1}^p\mathsf{H}_k(h,\cdot)\right\|_{L^{2}(\mathbb{T};\ell^2)}\leq C_2h^{p+1}\left(\sup_{t\in[0,h]}\|u(t)\|_{H^1}^{2p+2}+\sup_{t\in[0,h]}\|u(t)\|_{H^1}^{p+1}\right)
\end{align}
for some constant $C_2>0$ independent of $u$. Similarly, from Thm.~\ref{thm:interpolatory_magnus_quadrature_error} we find, using the NLS \eqref{eqn:NLS_in_hasimoto_transform} which shows that $\|\partial_{t}^j u\|_{L^2}\lesssim \|u\|_{H^{2j}}$ for all $j\geq0$, that
\begin{align}\label{eqn:quadrature_estimate_exact_NLS}
\left\|\mathsf{H}_k-\mathcal{Q}_h\left[\mathsf{H}_k;\mathbf{u}\right]\right\|_{L^{2}(\mathbb{T};\ell^2)}\leq C_3h^{p+1}\left\|u\right\|_{H^{2p+1}}^{2k},
\end{align}
for a constant $C_3>0$ independent of $u$.

\noindent Finally, we note that by construction of the quadrature $\mathcal{Q}$ in section~\ref{sec:time_integration_hasimoto_transform} we have
\begin{align*}
	&\left\|\mathcal{Q}_h\left[\mathsf{H}_k;\mathbf{u}\right](x)-\mathcal{Q}_h\left[\mathsf{H}_k;\mathbf{u}^{(0)}_{approx}\right](x)\right\|_{\ell^2}\\
	&\quad\quad\leq \sum_{\mathbf{l}}|b_{\mathbf{l}}|\left\|\mathcal{L}\left(\mathsf{A}\left(\left(\mathbf{u}\right)_{l_1}\right),\dots,\mathsf{A}\left(\left(\mathbf{u}\right)_{l_L}\right)\right)-\mathcal{L}\left(\mathsf{A}\left(\left(\mathbf{u}^{(0)}_{approx}\right)_{l_1}\right),\dots,\mathsf{A}\left(\left(\mathbf{u}^{(0)}_{approx}\right)_{l_L}\right)\right)\right\|_{\ell^2}\\
	&\quad\quad \leq C_4\sum_{\mathbf{l}}|b_{\mathbf{l}}|{\sum_{j=1}^{L}}\left(\left|\left(\mathbf{u}\right)_{l_j}-\left(\mathbf{u}^{(0)}_{approx}\right)_{l_j}\right|\left|\left(\mathbf{u}\right)_{l_j}+\left(\mathbf{u}^{(0)}_{approx}\right)_{l_j}\right|+\left|\left(\partial_x\mathbf{u}\right)_{l_j}-\left(\partial_x\mathbf{u}^{(0)}_{approx}\right)_{l_j}\right|\right)
\end{align*}
where $C_4>0$ is a constant depending on $\sup_{t\in[t_0,t_{1}]}\|u(t)\|_{H^{2p+1}}$. Note now that the $b_{\mathbf{l}}$ are uniformly bounded by $C_5h$ for some constant $C_5$ independent of $h$ (but depending on $p$) by the expression in \eqref{eqn:definition_of_b_l_quadrature_weights}. Taking the $L^2$-norm on both sides and using {Cor.~\ref{cor:convergence_thm_splitting_NLS_variable_time_step}} we conclude that
\begin{align}\label{eqn:estimate_quadrature_exact_vs_approximate}
\!\!\!\!\sum_{k=1}^p\left\|\mathcal{Q}_h\left[\mathsf{H}_k;\mathbf{u}\right](x)-\mathcal{Q}_h\left[\mathsf{H}_k;\mathbf{u}^{(0)}_{approx}\right](x)\right\|_{L^{2}(\mathbb{T};\ell^2)}&\leq {hC_6 \max_{1\leq j\leq L}\left\|\left(\mathbf{u}\right)_{l_j}-\left(\mathbf{u}^{(0)}_{approx}\right)_{l_j}\right\|_{H^1}}\leq C_7 h^{p+1}.
\end{align}
for constants $C_6,C_7$ which depend on $p, \sup_{t\in[t_0,t_{1}]}\|u(t)\|_{H^{2p+1}}$. The result then follows from \eqref{eqn:Lipschitz_estimate}-\eqref{eqn:estimate_quadrature_exact_vs_approximate}.

\end{proof}

\begin{cor}[Global convergence] Let $\mathbf{T}_m$ be the first row of $\mathsf{y}_m$ (computed using Scheme A). If  $\mathbf{T}:[0,T]\times \mathbb{T}\rightarrow \mathbb{S}^2$ is the exact solution to the SM equation \eqref{eqn:SMEqn} then we have, for every $m=0,\dots, M$:
	\begin{align*}
		\left\|\mathbf{T}_m-\mathbf{T}(mh)\right\|_{L^{2}\left(\mathbb{T};\ell^2\right)}\leq C h^{p},
	\end{align*}
	where $C$ depends on $p$ and $\sup_{t\in[0,T]}\|u(t)\|_{H^{2p+1}}$.
\end{cor}
\begin{proof}
	By a standard Lady Windermere's Fan argument combining Prop.~\ref{prop:stability_full_method} and Thm.~\ref{thm:local_error_high_regularity}.
\end{proof}

\section{Low regularity: A symmetric low-regularity integrator for the SM equation}\label{sec:low_regularity_integrator}

Having constructed a general set of novel integrators for the SM equation, we notice that in order to achieve first order convergence ($p=1$) our above methods would require initial datum $u_0\in H^{3}$. For highly regular initial data this is generally not a problem. However, we are also interested in low-regularity solutions of the SM equation such as for instance in the examples of \cite{banica2013,de2009numerical,gamayun2019domain,misguich2019domain}. For these cases it is desirable to reduce the regularity requirements for convergence and as such we will describe and study in this present section a dedicated low-regularity integrator reducing significantly the requirements of the above methodology. We will again exploit the Hasimoto transform \eqref{eqn:temporal_hasimoto_transform} but incorporate more of the structure of the NLS solution to \eqref{eqn:NLS_in_hasimoto_transform} in order to guarantee convergence for lower regularity regimes.

\subsection{A symmetric low-regularity numerical integrator for the NLS}
A central limiting factor in the convergence analysis of section~\ref{sec:convergence_analysis_cfl_high_regularity_integrators} was given by the regularity requirements of the splitting method for the solution of the NLSE part \eqref{eqn:NLS_in_hasimoto_transform}. Indeed, independently of our application to the SM equation, the low-regularity solution of the NLSE has been studied extensively over the recent decade and has resulted in the development of some highly efficient techniques \cite{ostermann2018low,bruned2022resonance}.

Motivated by the state-of-the-art integrator described in \cite{ostermann2018low} we introduce a \textit{new symmetrized} low-regularity integrator for the NLSE \eqref{eqn:NLS_in_hasimoto_transform}, which matches regularity requirements of previous work, but at the same time introduces structure preservation in this low-regularity regime. {Further details of the derivation of this {and similar as well as higher-order symmetric low-regularity integrators can be found in upcoming work \cite{alamabrunedmaierhoferschratz23} (see also \cite{AlamaBronsard23})} and as such we will limit ourselves in the present work to introducing the method and studying its convergence properties to the extend required for the convergence analysis of the SM approximation. The method can be written in the following form:
\begin{align}\label{eqn:symmetric_resonance_method}
	u^{m+1}&=e^{ih {\partial_{x}^2}}u^m+i\frac{h}{4}e^{ih{\partial_{x}^2}}\left[\left(u^{m}\right)^2\varphi_1(-ih{\partial_{x}^2})\overline{u^m}\right]+i\frac{h}{4}\left[\left(u^{m+1}\right)^2\varphi_1(ih{\partial_{x}^2})\overline{u^{m+1}}\right],
\end{align}
{where
\begin{align*}
    \varphi_1(z):=\begin{cases}
        \frac{e^z-1}{z},& z\neq0,\\
        1, & z=0,
    \end{cases}
\end{align*} and the action of the maps $\exp(ih{\partial_{x}^2}), \varphi_1(\pm i h\partial_x^2)$ (defined through functional calculus) can be efficiently evaluated by a simple diagonal operation in our spectral spatial discretisation.} {Additionally, as we show in Theorem~\ref{thm:implicit_NLS_integrator_fixed_pt_iterations}, fixed-point iteration can be used to solve the nonlinear system in \eqref{eqn:symmetric_resonance_method} at every time step, whereby the size of $\tau$ required for convergence is restricted solely by the magnitude of $\|u^{n}\|_{H^s}$ and entirely independent of the number of modes in the spatial discretisation.} Clearly, the above method is symmetric in the sense of definition \ref{def:time-symmetric_method}, and moreover it satisfies the following low-regularity convergence result:

{\begin{theorem}\label{thm:low_regularity_convergence_thm_NLS} Fix $s>1/2$ (corresponding to the norm in which we measure convergence) and $\gamma\in(0,1]$ (corresponding to the convergence order). If \eqref{eqn:NLS_in_hasimoto_transform} has a solution $u\in\mathcal{C}^0(0,T;H^{s+\gamma}(\mathbb{T}))$ then $\exists$ $C,h_0>0$ depending on $s,\gamma, \sup_{t\in[0,T]}\|u\|_{H^{s+\gamma}}$, such that for any $0< h<h_0$ we have
\begin{align*}
			\left\|u^m-u(t_m)\right\|_{H^s}\leq Ch^\gamma, \quad \forall\, 0\leq m\leq \left\lfloor\frac{T}{h}\right\rfloor,
		\end{align*}
  where $t_m=h m, m=0,\dots,\lfloor T/h\rfloor,$ and $u^m$ is computed using the method \eqref{eqn:symmetric_resonance_method}.
\end{theorem}}
\begin{proof}
	See Appendix~\ref{app:proof_of_convergence_thm_splitting_NLS}.
\end{proof}

\subsection{Fast low-regularity Hasimoto (FLowRH) transform}\label{sec:low_reg_hasimoto_transfrom}
Having introduced a symmetric low-regularity integrator for the NLS \eqref{eqn:NLS_in_hasimoto_transform} we turn to designing a low-regularity approximation to the temporal Hasimoto transform \eqref{eqn:temporal_hasimoto_transform}. We observe in Thms.~\ref{thm:truncated_magnus_error} \& \ref{thm:interpolatory_magnus_quadrature_error} that the quadrature of the terms in the Magnus expansion appears to impose more stringent regularity requirements than the truncation, thus we seek to replace the quadrature in \eqref{eqn:multivariate_quadrature_extension} with a tailored version. We will for this part of our work be content with designing a method of low order with guaranteed convergence under low regularity assumptions, thus it suffices to truncate the Magnus expansion after its first term. Thus we aim to find an approximation to the term
$\mathsf{H}_{1}(h,x)=\int_{0}^{h}\mathsf{A}(s)\dd s$. In order to do so we turn to the twisted variable (cf. \cite{ostermann2018low}) given by $v(t)=\exp(-it{\partial_{x}^2})u(t)$ which satisfies the twisted equation
	\begin{align}\label{eqn:twisted_variable_equation}
	i\partial_t v=\frac{1}{2}e^{-it{\partial_{x}^2}}\left[\left|e^{it{\partial_{x}^2}}v\right|^2e^{it{\partial_{x}^2}}v\right].
\end{align}
In order to approximate $\mathsf{H}_{1}(h,x)$ we thus have to, according to \eqref{eqn:temporal_hasimoto_transform}, compute approximations to the following two integrals:
\begin{align*}
	I_1[u;t_m,t_{m+1}]&=\int_{t_m}^{t_{m+1}}u_x(s)\dd s,\quad\quad
	I_2[u;t_m,t_{m+1}]=\int_{t_m}^{t_{m+1}} |u|^2\dd s.
\end{align*}
{Our quadrature rules for the above integrals will rely on evaluations of $u$, a solution to \eqref{eqn:NLS_in_hasimoto_transform}, at discrete points in time $t_m, 0\leq m\leq M$. Thus, as noted in the error estimates below we will throughout assume that we have at least $u\in \mathcal{C}^{0}(0,T;H^1(\mathbb{T}))$ for some $T>0$ and that our time evaluations are restricted to $t_m\in[0,T]$ for all $0\leq m\leq M$. This means that the quantities $u(t_m)$ and $\partial_xu(t_m)$ are well-defined for all $ 0\leq m\leq M$. Note also that the map $w\mapsto \exp(-it\partial_x^2)w$ is an isometry on $H^{s}$ for any $s\geq 0, t\in\mathbb{R}$. Thus, of course, if $u\in \mathcal{C}^{0}(0,T;H^s(\mathbb{T}))$ for some $s\geq 0$ then the twisted variable $v(t)=\exp(-it{\partial_{x}^2})u(t)$ has the same regularity, i.e. $v\in \mathcal{C}^{0}(0,T;H^s(\mathbb{T}))$.} We begin by approximating $I_1$: in terms of the twisted variable $v$ the integral can be written as
\begin{align*}
	I_1[u,t_m,t_{m+1}]=\int_{t_m}^{t_{m+1}}e^{is{\partial_{x}^2}}\partial_xv(s)\dd s.
\end{align*}
We choose a midpoint-type approximation of the form
\begin{align*}
	I_1[u,t_m,t_{m+1}]\approx \mathcal{Q}_1[u,t_m,t_{m+1}]&:=\int_{t_m}^{t_{m+1}} e^{i{\partial_{x}^2} s}\partial_x\frac{v(t_{m+1})+v(t_m)}{2}\dd s\\
	&=h\varphi_1(ih{\partial_{x}^2})\frac{e^{-i{\partial_{x}^2} h}\partial_xu(t_{m+1})+\partial_xu(t_m)}{2}.
\end{align*}
This can clearly be computed in $\mathcal{O}(N\log N)$ operations using the FFT in our spectral spatial discretisation. {Note, because we have constructed this approximation based on the \textit{twisted variable} $v(t)$ the map $e^{-i{\partial_{x}^2} h}$ applies only to the $u(t_{m+1})$ term in the final expression. This careful construction allows us to establish the following} error and stability estimates:
\begin{proposition}\label{prop:quadrature_error_Q1} If {$u\in\mathcal{C}^0(0,T;H^{1}(\mathbb{T}))$ is the} solution to \eqref{eqn:NLS_in_hasimoto_transform} then
	\begin{align*}
	\left\|I_1[u,t_m,t_{m+1}]- \mathcal{Q}_1[u,t_m,t_{m+1}]\right\|_{L^2}\leq h^2 C,
	\end{align*}
where $C$ depends on $\max_{t\in[t_m,t_{m+1}]}\|u(t)\|_{H^1}$. Moreover, for any {$u,w\in\mathcal{C}^0(t_m,t_{m+1};H^1(\mathbb{T}))$} we have
 \begin{align*}
 	\left\|\mathcal{Q}_1[u,t_m,t_{m+1}]-\mathcal{Q}_1[w,t_m,t_{m+1}]\right\|_{L^2}\leq Dh\left(\left\|u(t_{m+1})-w(t_{m+1})\right\|_{H^1}+\left\|u(t_{m})-w(t_{m})\right\|_{H^1}\right),
 \end{align*}
where $D>0$ is a constant independent of $u,w$.
\end{proposition}
\begin{proof}See Appendix~\ref{app:proof_quadrature_error_Q1}.	
\end{proof}

For the second integral we choose a slightly different approximation: We write the exact integral in terms of Fourier series:
\begin{align*}
	I_2[u;t_m,t_{m+1}]=\sum_{l\in\mathbb{Z}}e^{il x}\sum_{l=k_1-k_2}e^{-it_m(k_1^2-k_2^2)}\int_{0}^{h}e^{-is(k_1^2-k_2^2)} \hat{v}_{k_1}(t_m+s)\overline{\hat{v}_{k_2}(t_m+s)}\dd s,
\end{align*}
and define the following quadrature:

\ \vspace{-1cm}\ \begin{align*}
	\mathcal{Q}_2[u,t_m,t_{m+1}]&:=\sum_{l\in\mathbb{Z}}e^{il x}\left[\sum_{\substack{l=k_1-k_2\\k_1^2\geq k_2^2}}e^{-it_m(k_1^2-k_2^2)}\int_{0}^{h}e^{-isk_1^2}\dd s \frac{\hat{v}_{k_1}(t_{m+1})+\hat{v}_{k_1}(t_m)}{2} \frac{\overline{e^{-ih k_2^2}\hat{v}_{k_2}(t_{m+1})}+\overline{\hat{v}_{k_2}(t_m)}}{2}\right.\\
	&\quad\quad\quad\quad\quad+\left.\sum_{\substack{l=k_1-k_2\\k_1^2< k_2^2}}e^{-it_m(k_1^2-k_2^2)}\int_{0}^{h}e^{isk_2^2}\dd s \frac{e^{-ih k_1^2}\hat{v}_{k_1}(t_{m+1})+\hat{v}_{k_1}(t_m)}{2} \frac{\overline{\hat{v}_{k_2}(t_{m+1})}+\overline{\hat{v}_{k_2}(t_m)}}{2}\right]\\
	&=\sum_{l\in\mathbb{Z}}e^{il x}\left[\sum_{\substack{l=k_1-k_2\\k_1^2\geq k_2^2}}h\varphi_1(-ih k_1^2) \frac{e^{i h k_1^2}\hat{u}_{k_1}(t_{m+1})+\hat{u}_{k_1}(t_m)}{2} \frac{\overline{\hat{u}_{k_2}(t_{m+1})}+\overline{\hat{u}_{k_2}(t_m)}}{2}\right.\\
	&\quad\quad\quad\quad\quad\quad\quad\quad+\left.\sum_{\substack{l=k_1-k_2\\k_1^2< k_2^2}}h\varphi_1(ih k_2^2) \frac{\hat{u}_{k_1}(t_{m+1})+\hat{u}_{k_1}(t_m)}{2} \frac{\overline{e^{i h k_2^2}\hat{u}_{k_2}(t_{m+1})}+\overline{\hat{u}_{k_2}(t_m)}}{2}\right].
\end{align*}
It turns out that for a spectral discretisation this quadrature can be computed in $\mathcal{O}(N(\log N)^2)$ operations where $N$ is the number of spatial discretisation points, i.e. it is almost as quick as a fast Fourier transform. The details of this computation are provided in Appendix~\ref{app:fast_computation_of_index_restricted_convolutions}. The quadrature is a good choice because it has the following local error and stability property:
\begin{proposition}\label{prop:quadrature_error_Q2}
 If {$u\in\mathcal{C}^0(0,T;H^{1}(\mathbb{T}))$ is the} solution to \eqref{eqn:NLS_in_hasimoto_transform} then
\begin{align*}
	\left\|I_2[u,t_m,t_{m+1}]- \mathcal{Q}_2[u,t_m,t_{m+1}]\right\|_{L^2}\leq h^2 C,
\end{align*}
where $C>0$ depends on $\max_{t\in[t_m,t_{m+1}]}\|u(t)\|_{H^1}$. Moreover, for any {$u,w\in\mathcal{C}^0(t_m,t_{m+1};H^1(\mathbb{T}))$} we have:
\begin{align*}
\left\|\mathcal{Q}_2[u,t_m,t_{m+1}]- \mathcal{Q}_2[w,t_m,t_{m+1}]\right\|_{L^2}\leq D h \left(\left\|u(t_{m+1})-w(t_{m+1})\right\|_{H^1}+\left\|u(t_{m})-w(t_{m})\right\|_{H^1}\right),
\end{align*}
where $D>0$ depends on $\max\left\{\|u(t_m)\|_{H^1}+\|w(t_m)\|_{H^1},\|u(t_{m+1})\|_{H^1}+\|w(t_m)\|_{H^1}\right\}$.
\end{proposition}
\begin{proof}
See Appendix~\ref{app:proof_of_error_Q2}.
\end{proof}

Thus our overall approximation of $\mathsf{H}_1$ can be summarized in the quadrature rule
\begin{align}\label{eqn:overall_low-reg_quadrature_magnus}
	\mathcal{Q}^{(\text{low-reg})}_h[\mathsf{H}_1;u](x)=\begin{pmatrix}
		0&-\Im \mathcal{Q}_1[u,t_m,t_{m+1}]&\Re\mathcal{Q}_1[u,t_m,t_{m+1}]\\
		\Im\mathcal{Q}_1[u,t_m,t_{m+1}]&0&-\frac{1}{2}\mathcal{Q}_2[u,t_m,t_{m+1}]\\
		-\Re\mathcal{Q}_1[u,t_m,t_{m+1}]&\frac{1}{2}\mathcal{Q}_2[u,t_m,t_{m+1}]&0
	\end{pmatrix}
\end{align}
which can be computed using $\mathcal{O}(N(\log N)^2)$ operations (where $N$ is the number of Fourier modes in our spatial discretisation). We call the resulting discrete map $\mathsf{y}\mapsto \exp\left(\mathcal{Q}^{(\text{low-reg})}_h[\mathsf{H}_1;\mathbf{u}](x)\right)\mathsf{y}$ the fast low-regularity Hasimoto (FLowRH) transform. From Props.~\ref{prop:quadrature_error_Q1} \& \ref{prop:quadrature_error_Q2} we deduce
\begin{cor}\label{cor:error_full_quadrature_low_regularity_regime} If {$u\in\mathcal{C}^0(0,T;H^{1}(\mathbb{T}))$ is the} solution to \eqref{eqn:NLS_in_hasimoto_transform} then
	\begin{align*}
		\left\|\mathsf{B}(h)-\mathcal{Q}^{(\text{low-reg})}_h[\mathsf{H}_1;u]\right\|_{L^2(\mathbb{T};\ell^2)}\leq Ch^2,
	\end{align*}
for some constant $C$ depending on $\sup_{t\in[t_m,t_{m+1}]}\|u(t)\|_{H^1}$. Moreover, for any {$u,w\in\mathcal{C}^0(t_m,t_{m+1};H^1(\mathbb{T}))$} we have
\begin{align*}
	\left\|\mathcal{Q}^{(\text{low-reg})}_h[\mathsf{H}_1;u]-\mathcal{Q}^{(\text{low-reg})}_h[\mathsf{H}_1;w]\right\|_{L^2(\mathbb{T};\ell^2)}\leq D h\left( \left\|u(t_m)-w(t_m)\right\|_{H^1}+\left\|u(t_{m+1})-w(t_{m+1})\right\|_{H^1}\right),
\end{align*}
for some constant $D>0$ depending on $\max\left\{\|u(t_m)\|_{H^1}+\|w(t_m)\|_{H^1},\|u(t_{m+1})\|_{H^1}+\|w(t_m)\|_{H^1}\right\}$.
\end{cor}
{\begin{remark}
    In principle, higher order low-regularity Hasimoto transforms can be designed using similar ideas by including more terms in the Magnus series of $\mathsf{B}(t,x)$ and resolving the resulting integrals with bespoke quadrature rules similar to $\mathcal{Q}_1,\mathcal{Q}_2$ as designed above. For approximations up to second order (i.e. including terms up to $\mathsf{H}_2(t,x)$) this can be done using similar fast transforms as introduced in Appendix~\ref{app:fast_computation_of_index_restricted_convolutions} and the use of symmetric polynomial interpolants (cf. \cite[Section 3.2.1]{maierhoferschratz23}). However for terms involving at least three nested commutators such fast computations may no longer be possible, and new fast transform tools may have to be devised. In the interest of brevity this construction is not included here, but the development of a structured approach for the construction of higher order low-regularity Hasimoto transforms will form part of future research.
\end{remark}}
\subsection{Low-regularity integrator for SM equation and convergence analysis}
We can now write down our algorithm for the computation of $\mathbf{T}$ in the low-regularity regime, which will, as we show in Corollary~\ref{cor:global_convergence_low_regularity_regime} below, require only $u\in H^2$ for first order global convergence - much less than the $u\in H^3$ requirement of the algorithms described in section \ref{sec:CFL_free_high_regularity}.

\paragraph{Scheme B for the low-regularity approximation of $\mathbf{T}$:}

We use the notation introduced in the previous two sections. We fix $h$ the time step, and $T$ the end time and proceed as follows:

\begin{enumerate}
	\item From the given initial data $u_0(x), x\in\mathbb{T}$ we solve the NLS \eqref{eqn:NLS_in_hasimoto_transform} using the symmetric low-regularity integrator introduced in \eqref{eqn:symmetric_resonance_method} at $M$-time steps given by $t_{m}=hm$, where $M=T/h$. We denote those approximate values by $u^{(approx)}_m, m=0,\dots M$.
	\item From the given initial frame $\mathbf{T}_0(x),\mathbf{e}_{1,0}(x),\mathbf{e}_{2,0}(x),\,x\in\mathbb{T},$ we apply the low-regularity Magnus integrator arising from the approximation given in \eqref{eqn:overall_low-reg_quadrature_magnus}, the FLowRH transform, i.e.
	\begin{align*}
		\mathsf{y}_{m+1}=\Phi_{m\to m+1}(\mathsf{y}_m):=\exp\left(\mathcal{Q}^{(\text{low-reg})}_h[\mathsf{H}_1;(u^{(approx)}_{m},u^{(approx)}_{m+1})]\right)\mathsf{y}_m,
	\end{align*}
where $\mathsf{y}_m\approx\mathsf{y}(t_m)$ is the approximation of the full frame at time $t_m$ to obtain an approximation of $\mathbf{T}(t_m,x)$ at all time values $t_m=mh, 1\leq m\leq M$.
\end{enumerate}

\subsubsection*{Convergence analysis of the full method}\label{sec:convergence_results_low_reg_full_method}
The stability estimate Prop.~\ref{prop:stability_full_method} applies also in this case as no regularity assumptions were made. For the local error we have the following important result. Here we let again $\phi_{t_0\to t_1}:\mathbb{R}^3\rightarrow \mathbb{R}^3$ be the exact flow of the SM equation \eqref{eqn:SMEqn}, and let $\Phi^{(\text{low-reg})}_{m\to m+1}:\mathbb{R}^3\rightarrow \mathbb{R}^3$ be the map corresponding to our method as described in Scheme B above.

\begin{theorem}[Local error]\label{thm:local_error_low_reg_overall_method}
	Let $\gamma\in(1/2,1]$, $0\leq m\leq M-1$, and $R=\sup_{t\in[t_m,t_{m+1}]}\|u\|_{H^{1+\gamma}}$. Then there exists $h_0>0$ and $C>0$ depending on $\gamma, R$ such that for any $\mathsf{y}_0:\mathbb{T}\rightarrow O(n;\mathbb{R})$ (i.e. a function taking values in the set of orthogonal matrices) and any $0<h<h_0$ we have
	\begin{align*}
	\left\|\Phi_{m\to m+1}\left(\mathsf{y}_0\right)-\phi_{t_m\to t_{m+1}}\left(\mathsf{y}_0\right)\right\|_{L^{2}\left(\mathbb{T};\ell^2\right)}\leq C h^{1+\gamma}.
	\end{align*}
\end{theorem}
\begin{proof}This result follows analogously to Thm.~\ref{thm:local_error_high_regularity} by replacing the estimates from Thm.~\ref{thm:convergence_thm_splitting_NLS} with Thm.~\ref{thm:low_regularity_convergence_thm_NLS} and the ones from Thm.~\ref{thm:interpolatory_magnus_quadrature_error} by Corollary~\ref{cor:error_full_quadrature_low_regularity_regime}.
\end{proof}
From Thm.~\ref{thm:local_error_low_reg_overall_method} and Prop.~\ref{prop:stability_full_method} we obtain the following global convergence result of Scheme B:
\begin{cor}[Global convergence]\label{cor:global_convergence_low_regularity_regime} Let $\gamma\in(1/2,1]$, {suppose $u\in\mathcal{C}^0(0,T;H^{1+\gamma}(\mathbb{T}))$ is the solution to \eqref{eqn:NLS_in_hasimoto_transform} and $\mathbf{T},\mathbf{e}_1,\mathbf{e}_2\in\mathcal{C}^0(0,T;H^{1}(\mathbb{T}))$ is the solution to \eqref{eqn:3x3_matrix_formulation_temporal_hasimoto}. Let $R=\sup_{t\in[0,t_{M}]}\|u\|_{H^{1+\gamma}}$ and denote} by $\mathbf{T}_m$ the first row of $\mathsf{y}_m$ (computed using Scheme B). Then there exists $h_0>0, C>0$ depending on $\gamma, R$ such that for any $0<h<h_0$ and every $0\leq m\leq M$ we have
	\begin{align*}
		\left\|\mathbf{T}_m-\mathbf{T}(mh)\right\|_{L^{2}\left(\mathbb{T};\ell^2\right)}\leq C h^{\gamma}.
	\end{align*}
\end{cor}

\section{Numerical experiments}\label{sec:numerical_experiments}
Having understood the theoretical convergence properties we can now consider the performance of our proposed methods in practice. As reference for the state-of-the-art in the literature we use \cite{xie2020second} (semi-implicit unconditionally stable integrators for the SM equation) although important prior work is also given by \cite{weinan2001numerical}, \cite{de2009numerical} and \cite{buttke1988numerical}. Notably \cite{de2009numerical,buttke1988numerical} require stringend CFL-conditions for convergence (in the case of \cite{de2009numerical} to ensure stability and in the case of \cite{buttke1988numerical} to ensure solubility of the implicit equations).

\subsection{Smooth solutions to the SM equation}
In our first example we consider smooth initial conditions for the problem \eqref{eqn:SMEqn}. In particular we consider, motivated by the computational examples in prior work \cite{weinan2001numerical}, the initial condition
\begin{align}\label{eqn:smooth_initial_data}
    \mathbf{T}_0(x)=(\cos (2x)\sin(x),\sin(2x)\sin(x),\cos(x)).
\end{align}
In our numerical simulations we compared two versions of our Scheme A (corresponding to notation $p=2,3$ in terms of section \ref{sec:convergence_analysis_cfl_high_regularity_integrators}), and our scheme B with the first and second order algorithms introduced by Xie et al. \cite{xie2020second}. We note that the algorithm from \cite{xie2020second} is semi-implicit and as such requires, at each time step, the solution of a linear system which we did in our implementation using GMRES {with an analytical preconditioner}. All of our methods were implemented on an Intel(R) Core(TM) i7-10700 CPU @ 2.90GHz using 8 CPU cores where many calculations were parallelised (for instance the solution of the implicit equations in the method by \cite{xie2020second} is done in parallel as is step 2 in our schemes A and B). Finally, we note that in our implementation of \cite{xie2020second} we exploited the fact that we work on a periodic domain and compute all required derivatives using fast Fourier transform methods in exactly the same way as for our novel methods.

In the first instance we look at the convergence properties of the numerical schemes for a moderate spatial discretisation of size $N=256$. The reference solution was computed with $h=2^{-15}$ and $N=4096$ using our scheme A (order 4) and the results are shown in Fig.~\ref{fig:convergence_plots_smooth_N_128}. Clearly we see in Fig.~\ref{fig:convergence_plots_smooth_N_128_timestep} that all methods achieve the predicted theoretical convergence rates (note our scheme B is symmetric and thus at least of second order for smooth initial data). Furthermore, in Fig.~\ref{fig:convergence_plots_smooth_N_128_cputime} the computational advantage of the fully explicit nature of our Hasimoto transforms becomes apparent as all of our methods outperform the previous state-of-the-art by at least an order of magnitude in terms of computational time.

\begin{figure}[h!]
    \centering
    \begin{subfigure}{0.495\textwidth}
\includegraphics[width=0.95\textwidth]{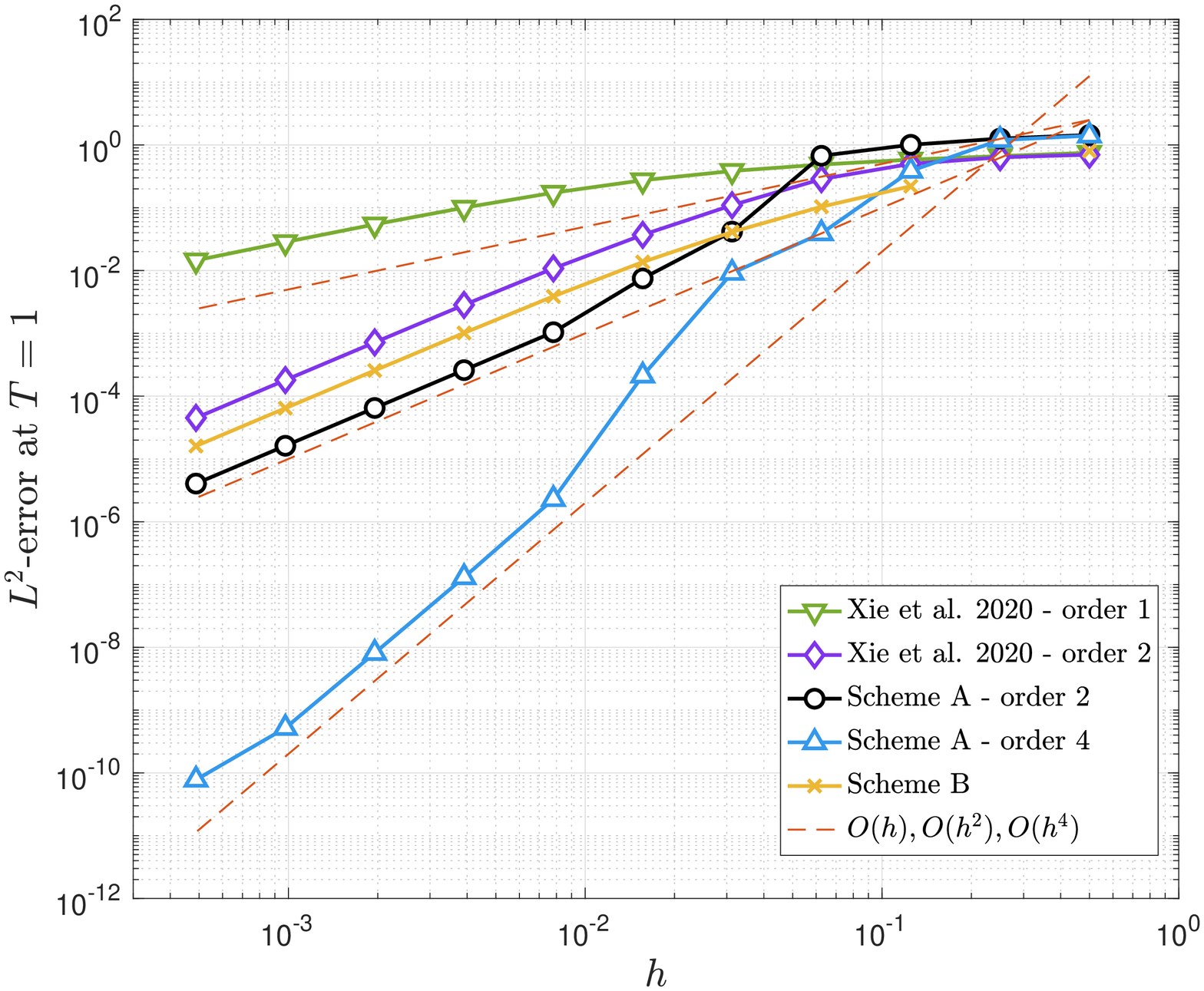}
		\caption{$L^2$-error in $\mathbf{T}$ at time $T=1$ versus the {step size}.}
  \label{fig:convergence_plots_smooth_N_128_timestep}
	\end{subfigure}
\begin{subfigure}{0.495\textwidth}
\includegraphics[width=0.95\textwidth]{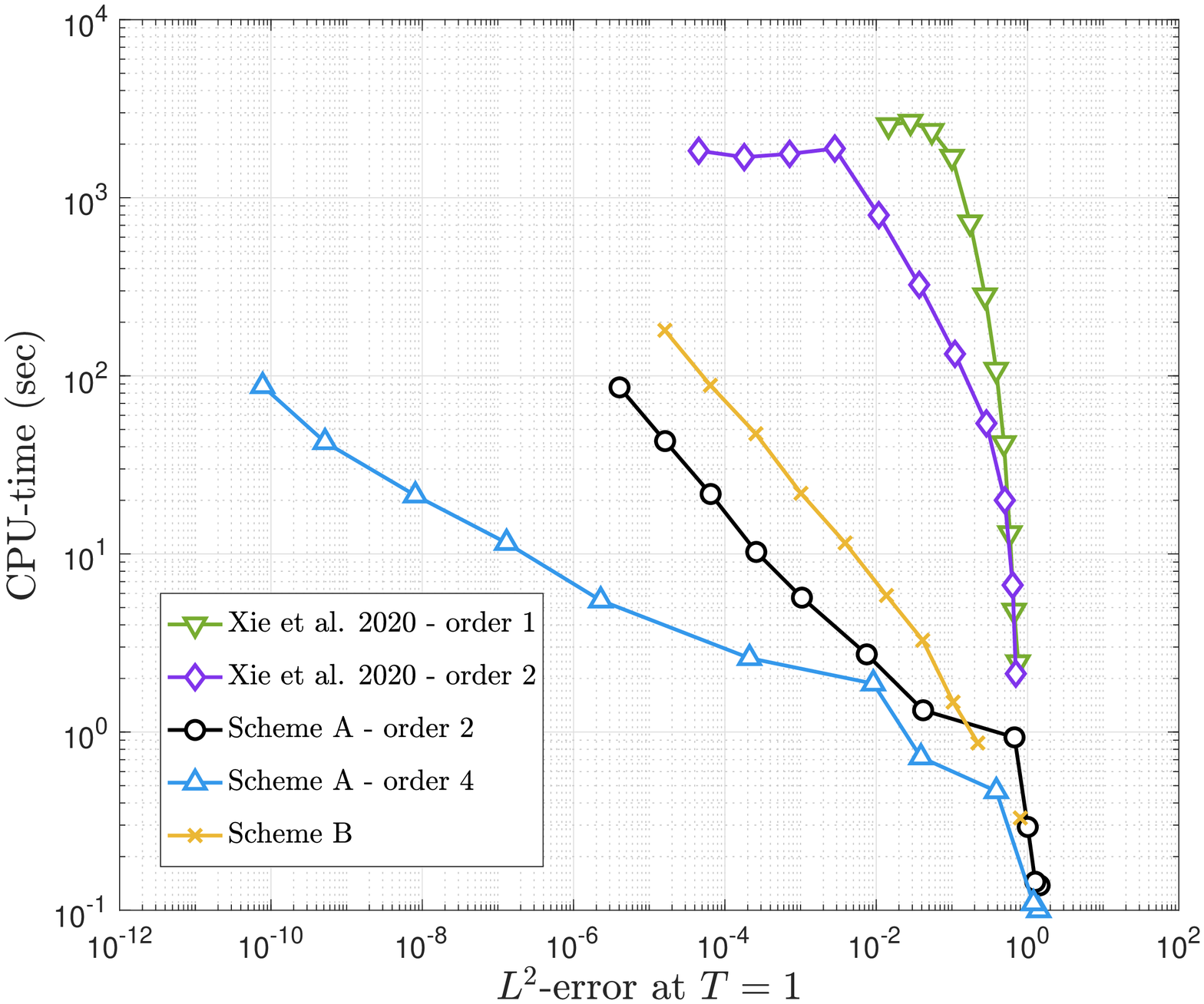}
		\caption{CPU time versus the $L^2$-error in $\mathbf{T}$ at time $T=1$.}
  \label{fig:convergence_plots_smooth_N_128_cputime}
	\end{subfigure}
 \caption{Convergence plots for smooth initial data \eqref{eqn:smooth_initial_data} and $N=256$.}
 \label{fig:convergence_plots_smooth_N_128}
\end{figure}

\begin{figure}[h!]
    \centering
    \begin{subfigure}{0.495\textwidth}
\includegraphics[width=0.95\textwidth]{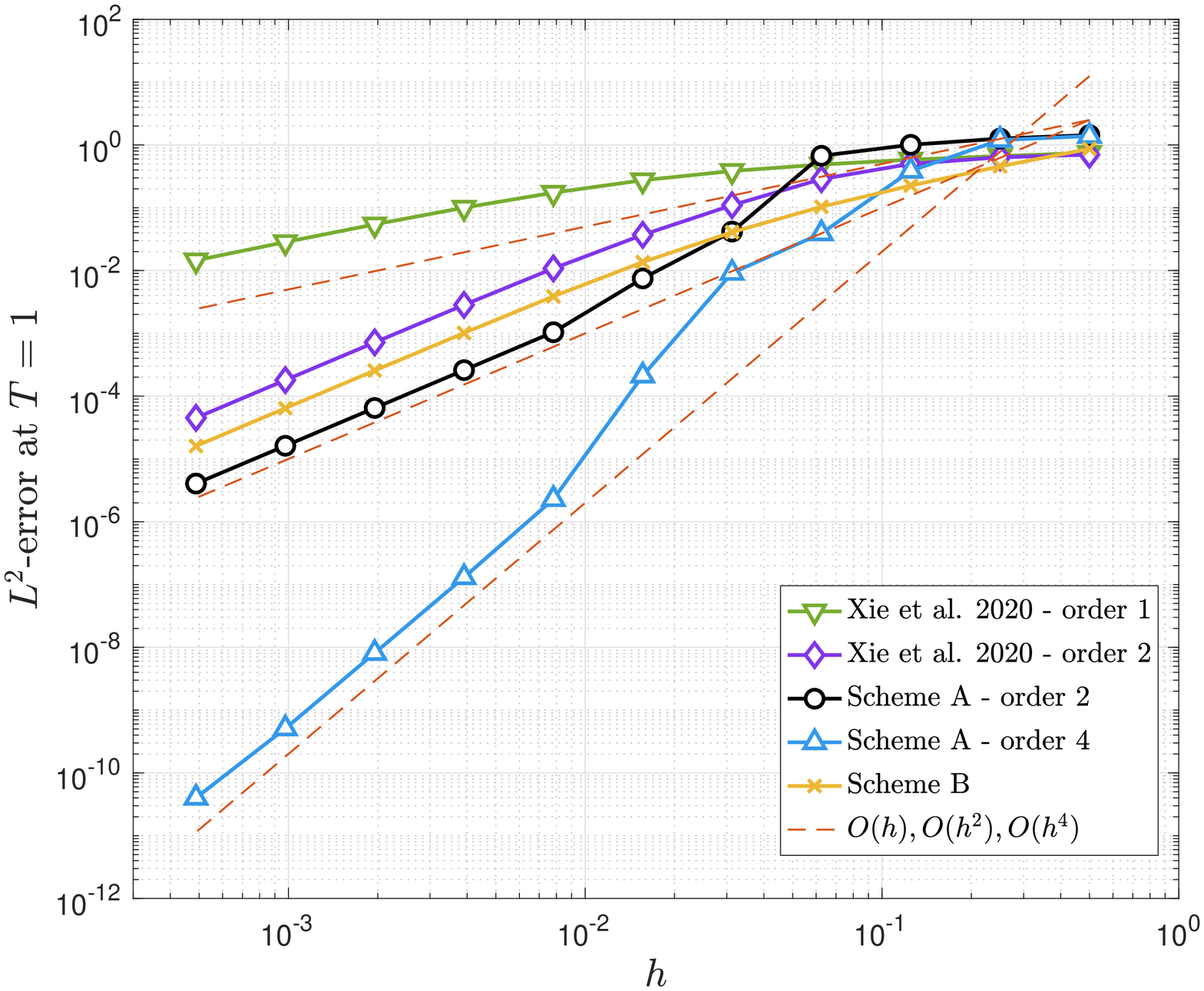}
		\caption{$L^2$-error in $\mathbf{T}$ at time $T=1$ versus the {step size}.}
  \label{fig:convergence_plots_smooth_N_1024_timestep}
	\end{subfigure}
\begin{subfigure}{0.495\textwidth}
\includegraphics[width=0.95\textwidth]{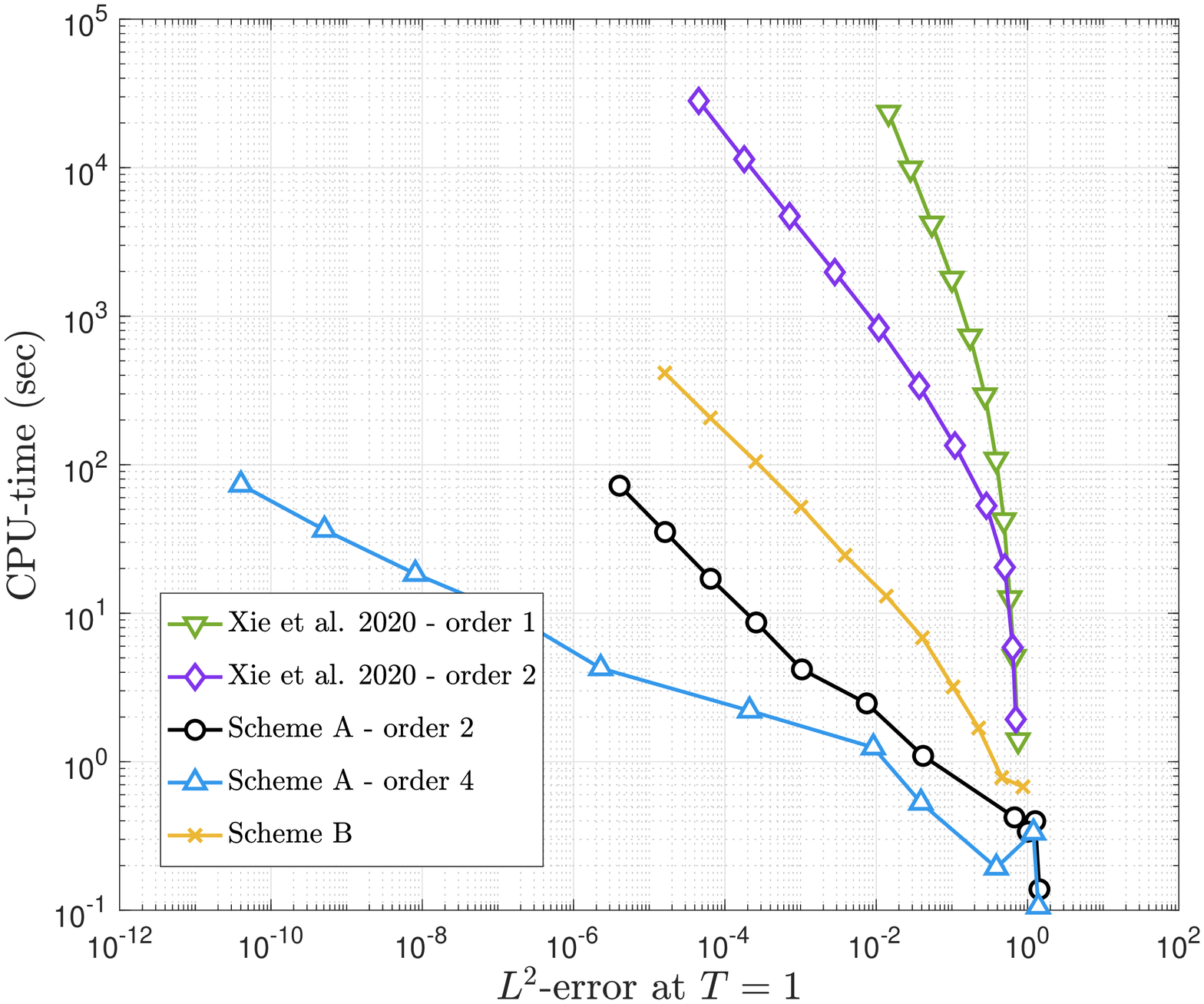}
		\caption{CPU time versus the $L^2$-error in $\mathbf{T}$ at time $T=1$.}
  \label{fig:convergence_plots_smooth_N_1024_cputime}
	\end{subfigure}
 \caption{Convergence plots for smooth initial data \eqref{eqn:smooth_initial_data} and $N=1024$.}
 \label{fig:convergence_plots_smooth_N_1024} 
\end{figure}

{Increasing the number of spatial discretisation points to $N=1024$ we see in Fig.~\ref{fig:convergence_plots_smooth_N_1024} that our new integrators are indeed} truly unconditionally stable and convergent and incur no significant increase in cost as $N$ increases (indeed as shown in earlier sections the cost of our scheme A scales like $\mathcal{O}(N\log N)$ and that of our scheme B scales like $\mathcal{O}(N(\log N)^2)$). {In particular, out of our new schemes, only scheme B involves an implicit aspect in \eqref{eqn:symmetric_resonance_method} which we show can be solved with fixed-point iteration at rates independent of $N$ (cf. Theorem~\ref{thm:implicit_NLS_integrator_fixed_pt_iterations}).} The reference solution in this experiment was again computed with $h=2^{-15}$ and $N=4096$ using our scheme A (order 4).

In the final Fig.~\ref{fig:structure_preservation_plots_smooth_N_128} we consider the absolute error in the $\mathcal{E},\mathcal{I}$ introduced in \eqref{eqn:definition_conserved_quantities_SM} for a fixed time step $h=1/200$ with spatial discretisation $N=128$. We observe that our symmetric methods indeed exhibit good preservation of these actions.
\begin{figure}[h!]
    \centering
    \begin{subfigure}{0.495\textwidth}
\includegraphics[width=0.95\textwidth]{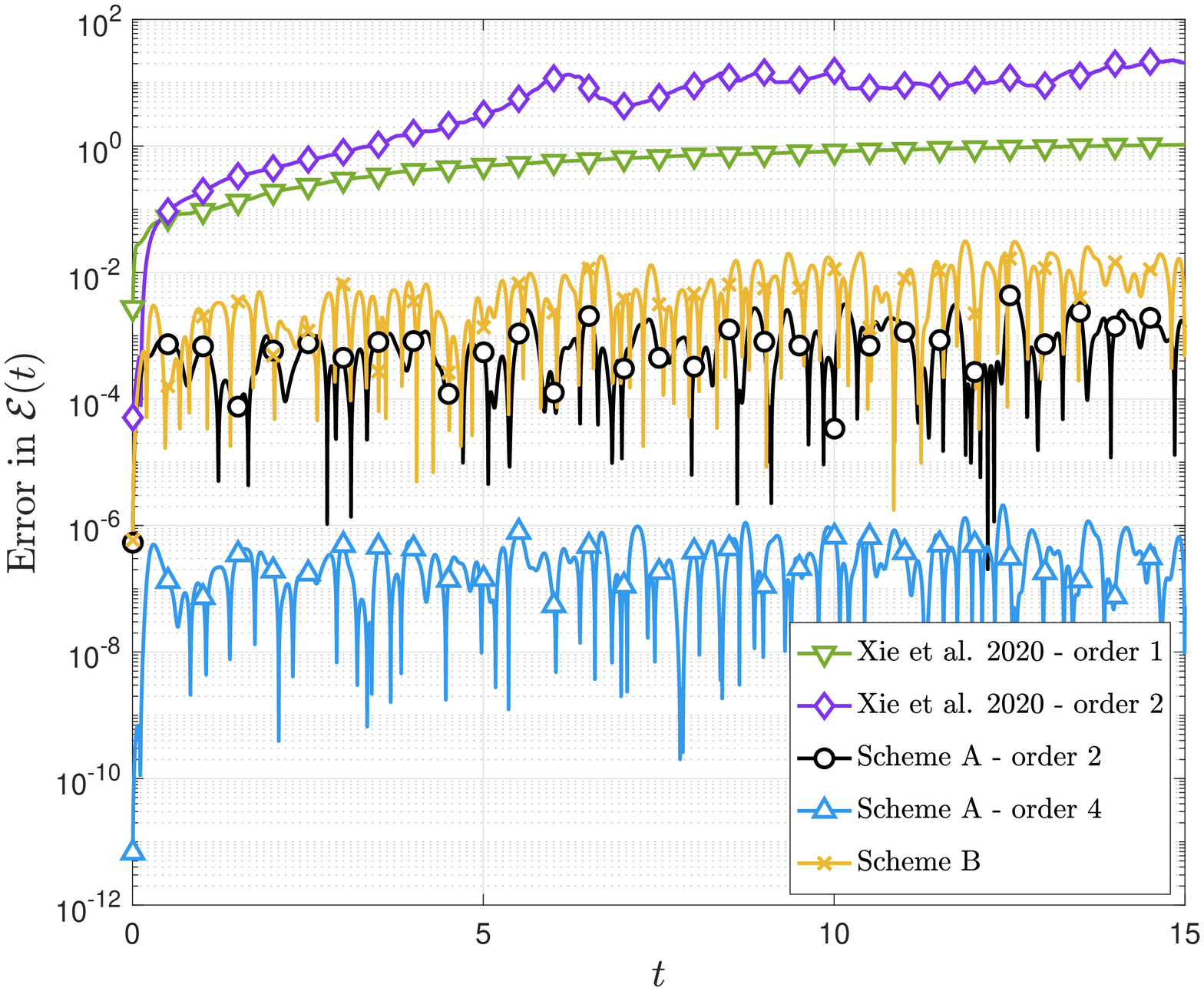}
		\caption{Absolute error in $\mathcal{E}(t)$.}
	\end{subfigure}
\begin{subfigure}{0.495\textwidth}
\includegraphics[width=0.95\textwidth]{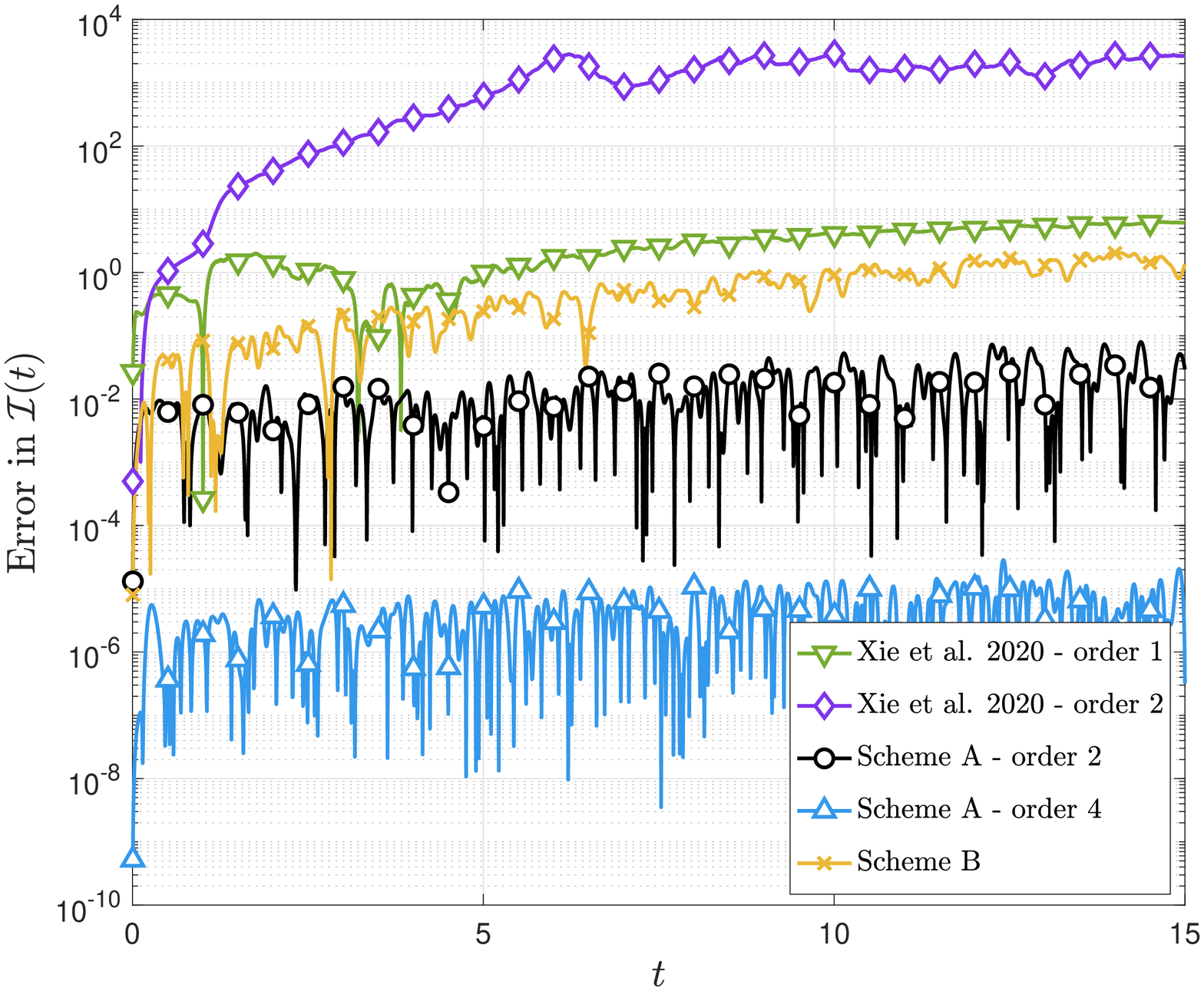}
		\caption{Absolute error in $\mathcal{I}(t)$.}
	\end{subfigure}
\caption{Error in the approximation of the conserved quantities for $\tau=1/200, N=128$ and initial conditions \eqref{eqn:smooth_initial_data}.}
 \label{fig:structure_preservation_plots_smooth_N_128} 
\end{figure}
\vspace{-0.5cm}
\subsection{Rough SM-solutions and application to the vortex filament equations}\label{sec:application_to_VFE}
As a second example we consider the evolution of a closed low-regularity vortex filament in an ideal fluid. The filament $\mathbf{X}(t):\mathbb{T}\rightarrow\mathbb{R}^3$ evolves according to the vortex filament equation (VFE)
\begin{align*}
	\partial_t\mathbf{X}=\partial_x\mathbf{X}\times\partial_{x}^2\mathbf{X}.
\end{align*}
The tangent vector field $\mathbf{T}=\partial_x\mathbf{X}$ satisfies the SM equation \eqref{eqn:SMEqn}. The VFE has been the subject of several numerical studies, amongst them \cite{buttke1988numerical} and \cite{de2009numerical}. In the present example we consider a low-regularity initial vortex filament which is flat and composed of two straight sections that are connected by half-circles as indicated in Fig.~\ref{fig:evolution_of_low-reg_vortex_filament} (which also features the time evolution of this vortex filament {with $h=0.01$ and for two different spatial discretisations}).

\begin{figure}[h!]
    \centering
    \begin{subfigure}{0.495\textwidth}
\includegraphics[width=1.0\textwidth]{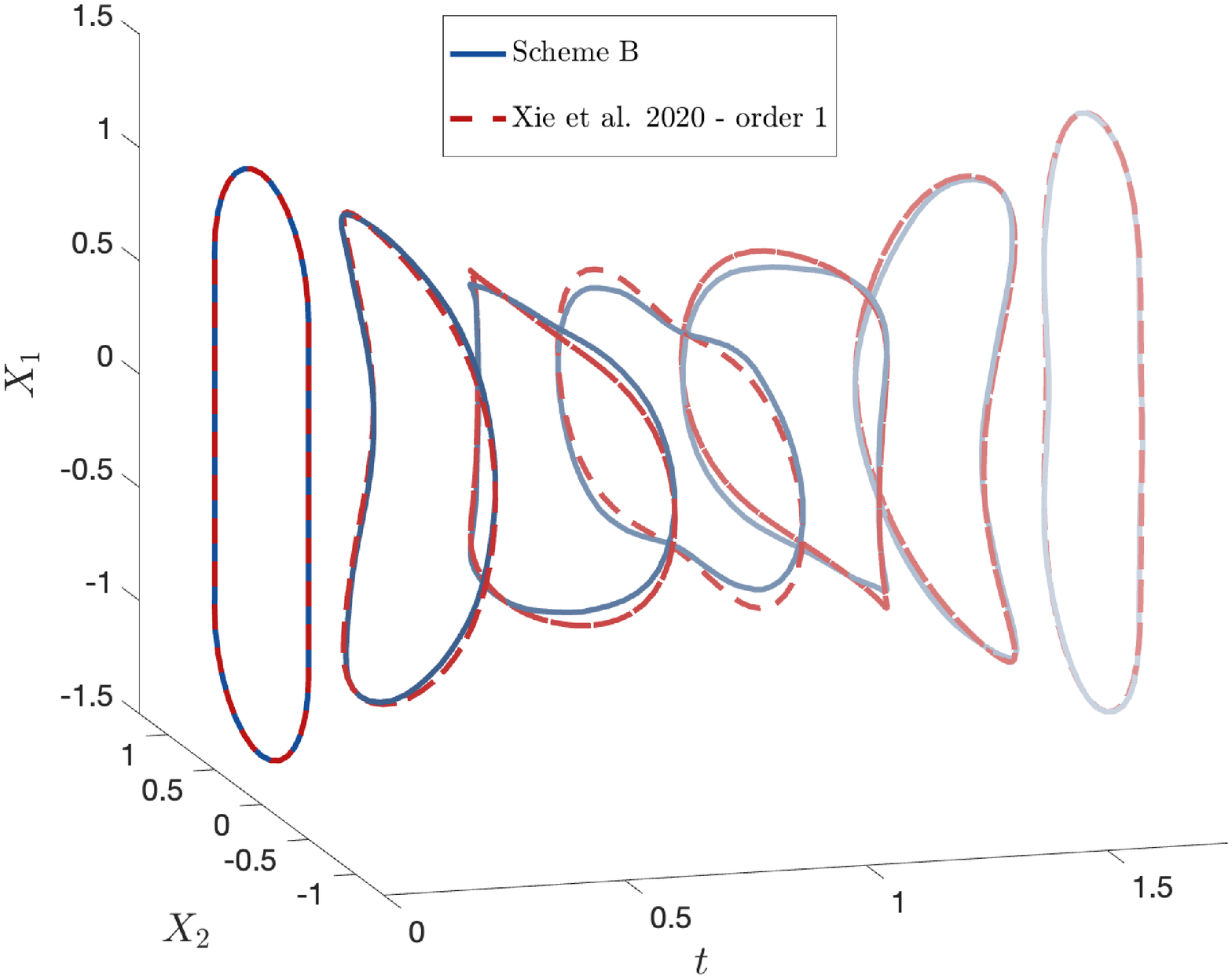}
		\caption{$N=64$.}
	\end{subfigure}
\begin{subfigure}{0.495\textwidth}
\includegraphics[width=1.0\textwidth]{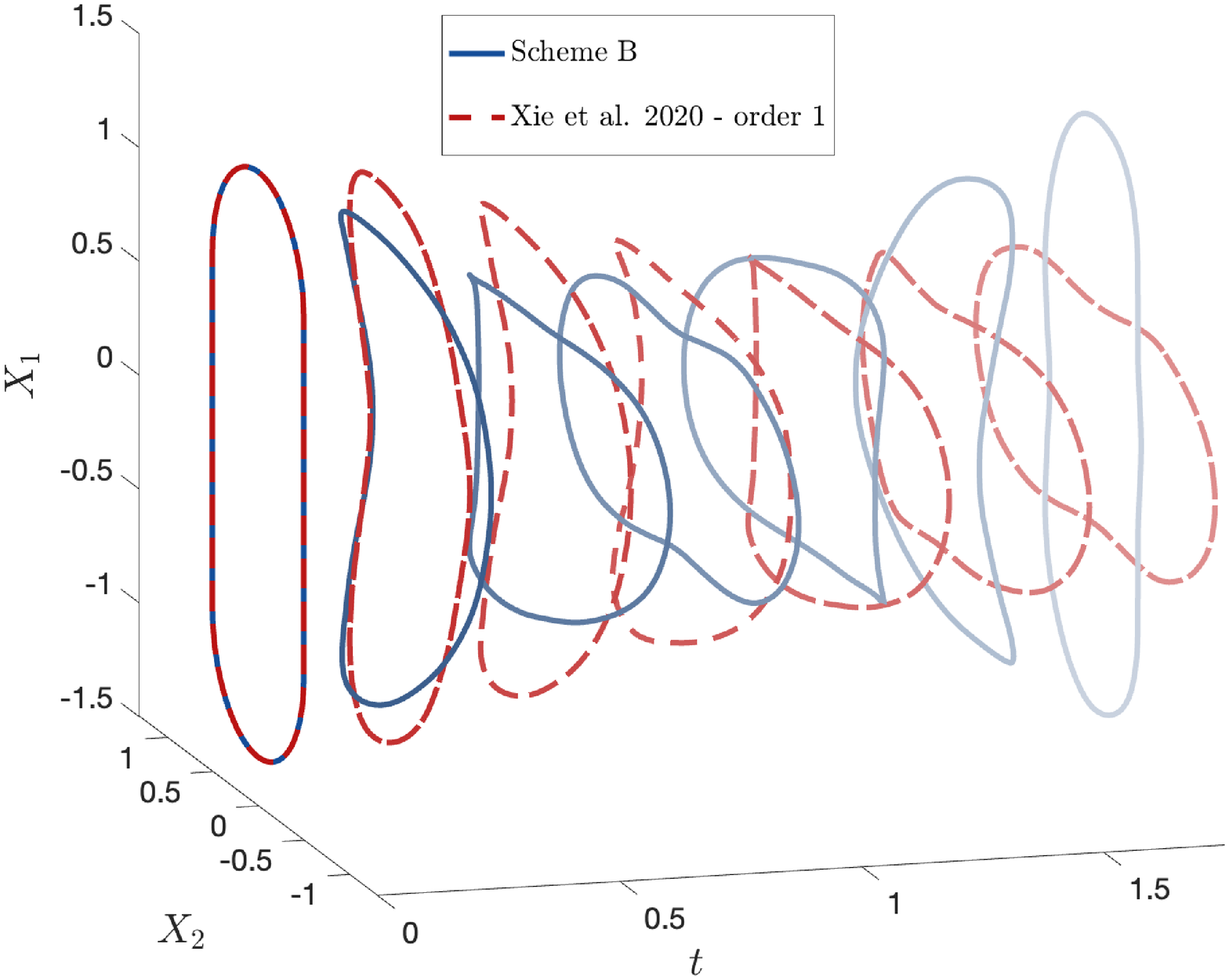}
		\caption{$N=256$.}
	\end{subfigure}
 \caption{Evolution of low-regularity vortex filament {for $h=0.01$}.}
 \label{fig:evolution_of_low-reg_vortex_filament}
\end{figure}

The tangent field initially takes the form
\begin{align}\label{eqn:low-reg_initial_data}
  \mathbf{T}_0(x)=\begin{cases}(\cos(2x+\pi),\sin(2x+\pi),0),&\text{if\ }-\pi\leq x<-\pi/2,\\
  (1,0,0),&\text{if\ }-\pi/2\leq x<0,\\
  (\cos(2x),\sin(2x),0),&\text{if\ }0\leq x<\pi/2,\\
  (-1,0,0),&\text{if\ }\pi/2\leq x<\pi,
  \end{cases}
\end{align}
so the curvature of the corresponding planar curve (i.e. $\tau_0\equiv 0$) is piecewise constant, which means that ${\mathbf{T}_0\in H^{3/2-\epsilon}}$ and $u_0\in H^{1/2-\epsilon}$ for all $\epsilon >0$. {We recall the existence of weak solutions for this class of initial conditions by \cite{jerrard_smets_2012}.} This is less regularity than required in Cor.~\ref{cor:global_convergence_low_regularity_regime}, but we will demonstrate that in practical performance our low-regularity scheme B can still achieve reliable convergence in this regime, while {all of the classical methods suffer from significant order reduction}. This clearly underlines the important properties of our new method. {In the following experiments the reference solution is computed using scheme B with $h=2^{-15}$ and $N=4096$.} Indeed in our convergence graph {with $N=256$} in Fig.~\ref{fig:convergence_plots_rough_N_128} we observe that {the convergence behaviour of all classical methods and even our scheme A is significantly worse than in the smooth case,} while our scheme B is able to achieve reliable convergence at rate roughly $\mathcal{O}(h^{1/2})$ {until the spatial discretisation error becomes dominant and the error thus stagnates at roughly $0.05$}. {Although the other methods exhibit some convergence until the spatial discretisation error is attained the convergence appears to be much slower and more unreliable than in the smooth case.} {Note the spatial discretisation error was not visible in Fig.~\ref{fig:convergence_plots_smooth_N_128} and Fig.~\ref{fig:convergence_plots_smooth_N_1024} because it is much smaller in the smooth case due to the rapid decay of Fourier coefficients in the solution. {As we increase the number of spatial discretisation points to $N=1024$ (cf. Fig.~\ref{fig:convergence_plots_rough_N_1024}) the spatial discretisation error becomes smaller and the error curve for our scheme B is consequently lower, while for the remaining methods the effect of order reduction becomes even more pronounced and the latter appear to require even smaller values of $h$ to achieve a comparable degree of convergence.}} {The problem experienced with prior work, for example Xie et al. \cite{xie2020second}, in this low-regularity regime is also visible in Fig.~\ref{fig:evolution_of_low-reg_vortex_filament} which shows the detrimental effect an increase in $N$ has on the approximation for fixed $h$, while the novel scheme B is able to reliably approximate the solution even for large values of $N$ without deteriotion of its convergence properties with respect to $h$.}

\begin{figure}[h!]
    \centering
    \begin{subfigure}{0.495\textwidth}
\includegraphics[width=0.95\textwidth]{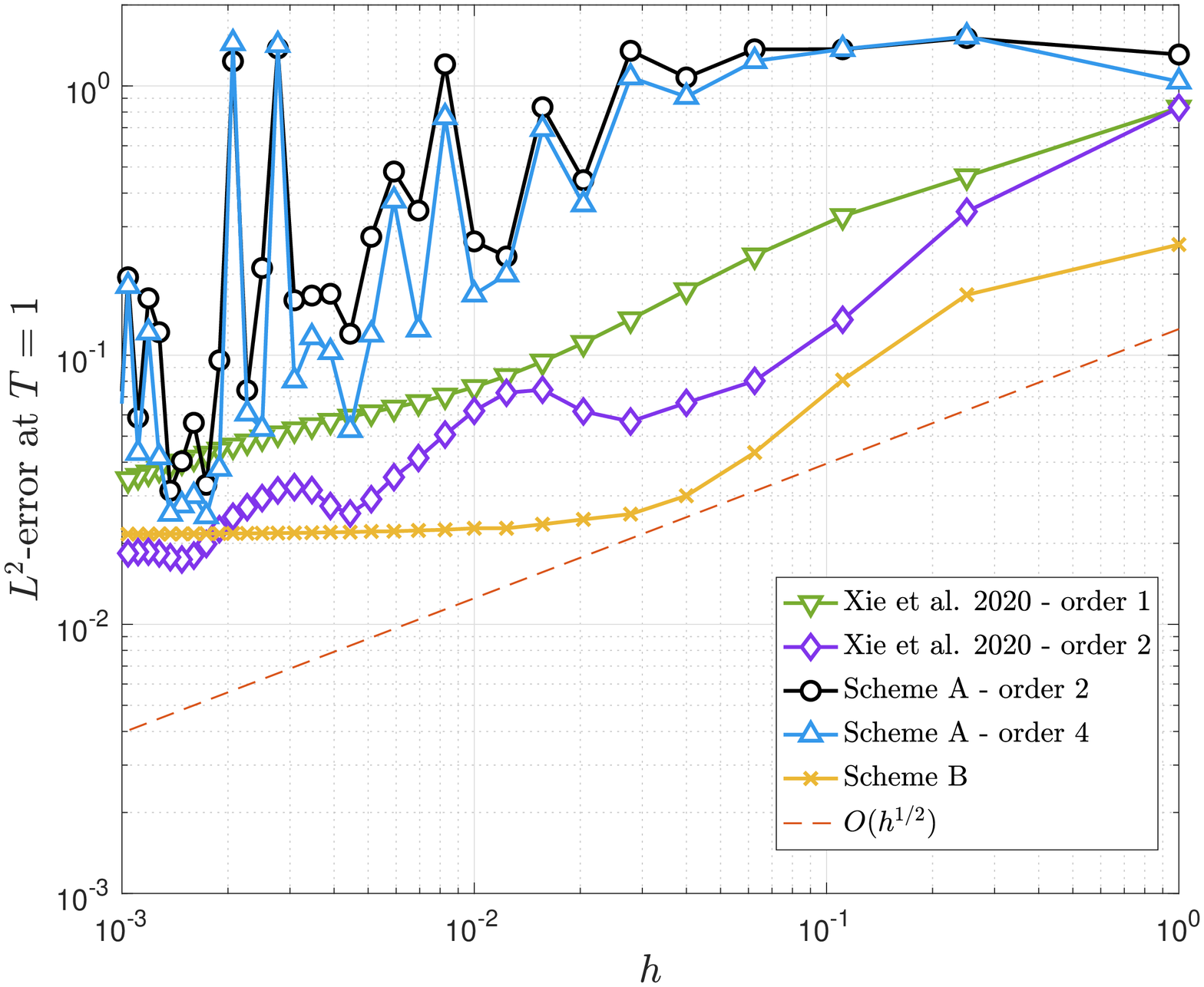}
		\caption{$L^2$-error in $\mathbf{T}$ at time $T=1$ versus the {step size}.}
  \label{fig:convergence_plots_rough_N_128_timestep}
	\end{subfigure}
\begin{subfigure}{0.495\textwidth}
\includegraphics[width=0.95\textwidth]{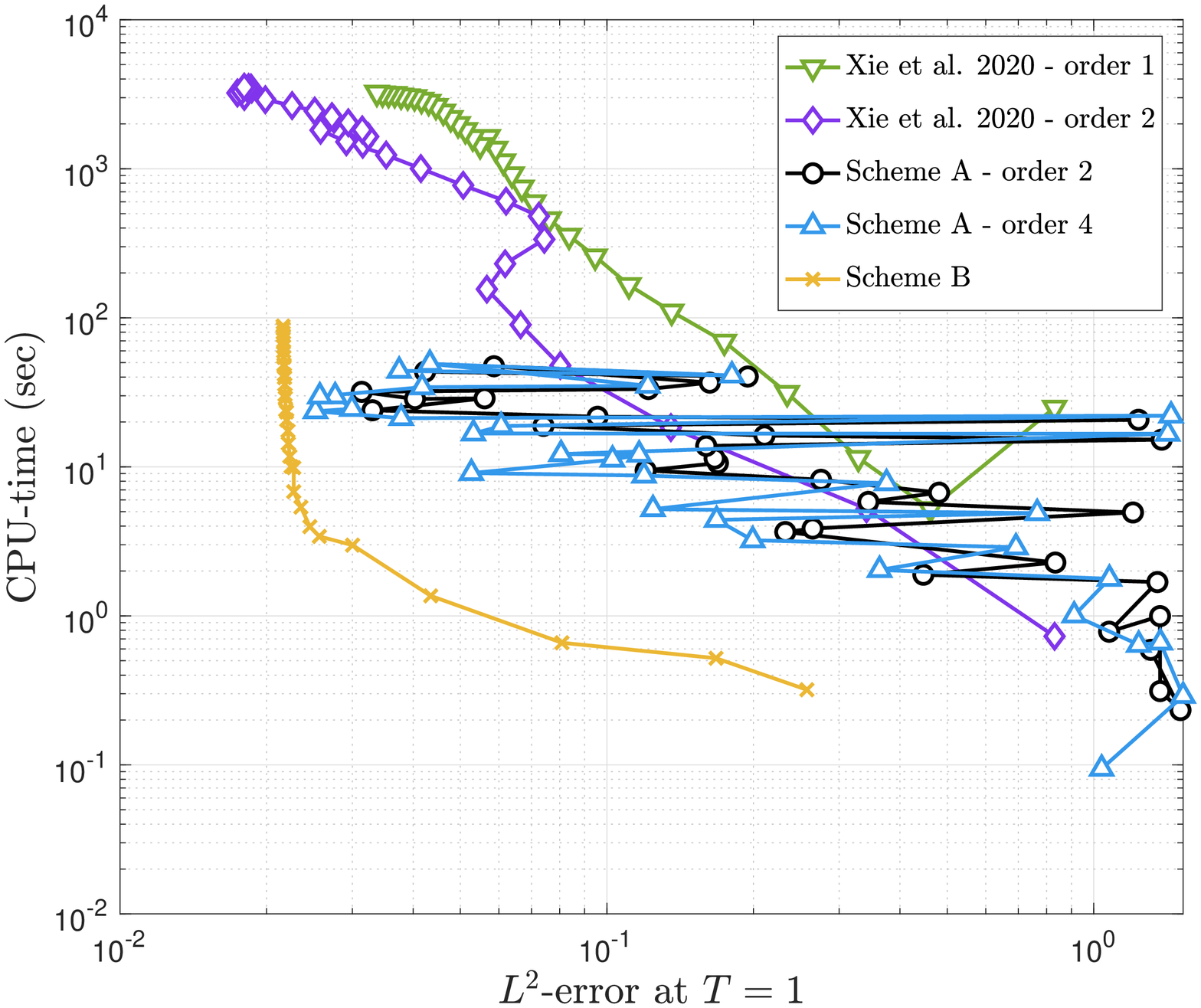}
		\caption{CPU time versus the $L^2$-error in $\mathbf{T}$ at time $T=1$.}
  \label{fig:convergence_plots_rough_N_128_cputime}
	\end{subfigure}
 \caption{Convergence plots for low-regularity initial data \eqref{eqn:low-reg_initial_data} and $N=256$.}
 \label{fig:convergence_plots_rough_N_128} 
\end{figure}

\begin{figure}[h!]
    \centering
    \begin{subfigure}{0.495\textwidth}
\includegraphics[width=0.95\textwidth]{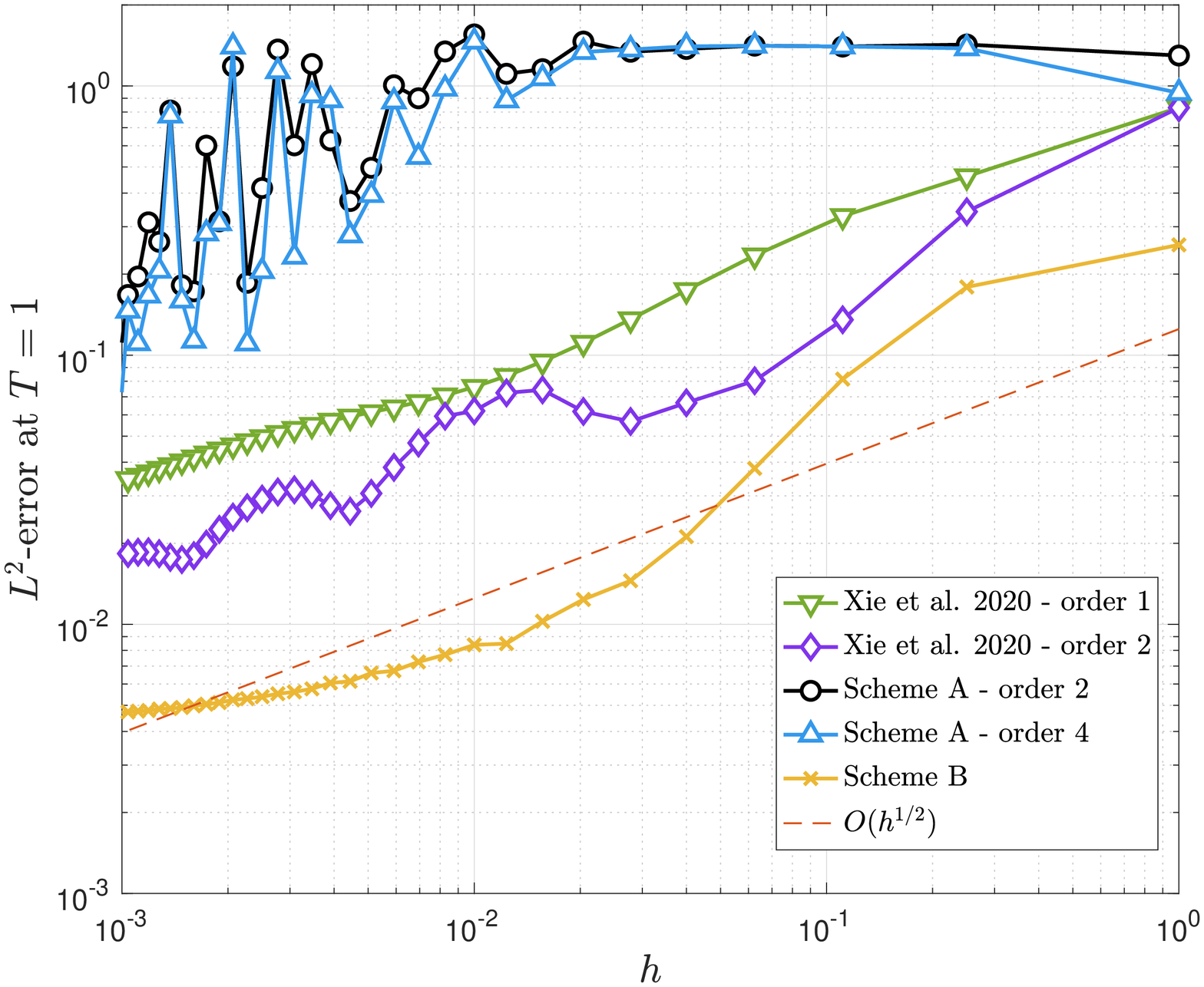}
		\caption{$L^2$-error in $\mathbf{T}$ at time $T=1$ versus the {step size}.}
  \label{fig:convergence_plots_rough_N_1024_timestep}
	\end{subfigure}
\begin{subfigure}{0.495\textwidth}
\includegraphics[width=0.95\textwidth]{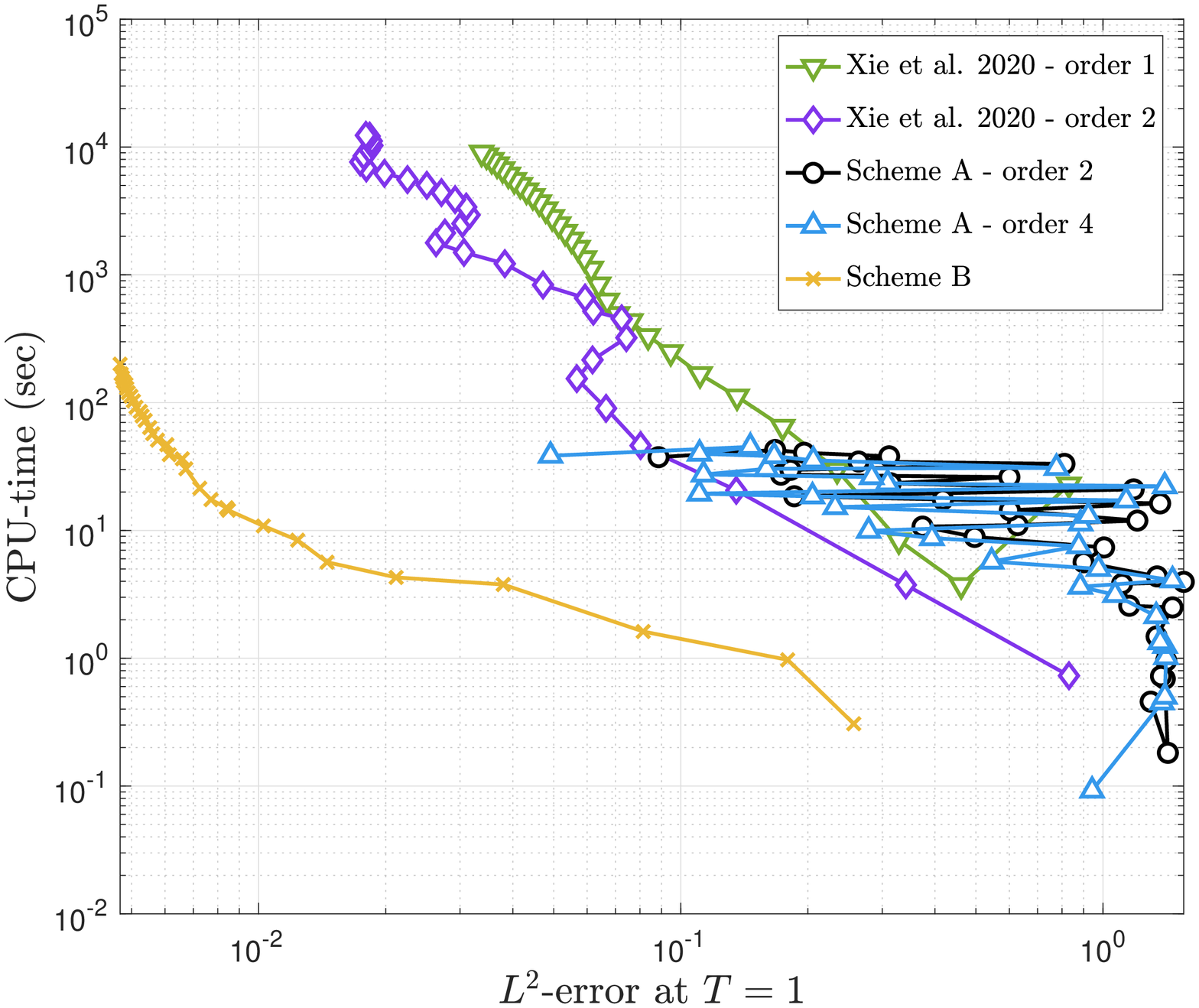}
		\caption{CPU time versus the $L^2$-error in $\mathbf{T}$ at time $T=1$.}
  \label{fig:convergence_plots_rough_N_1024_cputime}
	\end{subfigure}
 \caption{Convergence plots for low-regularity initial data \eqref{eqn:low-reg_initial_data} and $N=1024$.}
 \label{fig:convergence_plots_rough_N_1024} 
\end{figure}

Finally, we note that for this low-regularity regime we have ${\partial_{x}^2}\mathbf{T}\notin L^{2}$ so we cannot make sense of the action $\mathcal{I}$. Nevertheless, we can look at the absolute error in $\mathcal{E}$ and we find that our symmetric low-regularity integrator scheme B appears to preserve the action reasonably well too (cf. Fig.~\ref{fig:structure_preservation_plot_rough_N_128}).
\begin{figure}[h!]
    \centering
\includegraphics[width=0.45\textwidth]{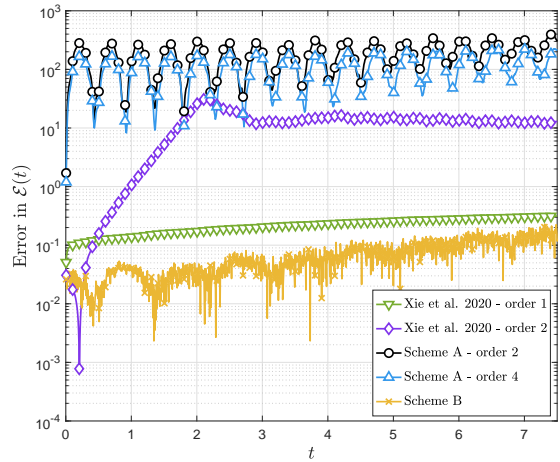}
		\caption{Absolute error in the approximation of $\mathcal{E}(t)$ for $h=1/200, N=128$.}
  \label{fig:structure_preservation_plot_rough_N_128} 
\end{figure}

\newpage\section{Concluding remarks}\label{sec:conclusions}

In this manuscript we presented a novel numerical approach to the SM equation by exploiting the Hasimoto transform. Using this nonlinear transform we were able to design the first fully explicit symmetric unconditionally stable numerical schemes for the SM equation with values in $\mathbb{S}^2$, which notably can be designed to arbitrary order. We also developed a tailored low-regularity integrator for the SM (including the FLowRH transform) which extends the computational possibilities to regimes which were not accessible with prior methods. Our rigorous error analysis and computational experiments demonstrate that our novel methodology is able to significantly outperform the current state-of-the-art.

As a possible direction of future work we note that our algorithm for the fast computation of index restricted convolution-type sums, which we described in Appendix~\ref{app:fast_computation_of_index_restricted_convolutions}, opens up a plethora of new possibilities for the design of low-regularity integrators of dispersive partial differential equations. Indeed, until now the design of such integrators (and in particular resonance-based schemes) was restricted to a composition of diagonal operations in Fourier and physical space together with a fast Fourier transform to switch between these pictures and the present concept of fast computation of index restricted convolution type expressions may extend more generally to a wider class of integrators for dispersive systems.
\vspace{-0.2cm}
\section*{Acknowledgements}
The authors would like to thank Daan Huybrechs (KU Leuven), Arieh Iserles (University of Cambridge), Maryna Kachanovska (INRIA Saclay) and Luis Vega (Basque Center for Applied Mathematics) for several interesting and helpful discussions. VB gratefully acknowledges support from the Institut Universitaire
de France membership and from the French ANR project SingFlows. {GM and KS gratefully acknowledge funding from the European Research Council (ERC) under the European
	Union’s Horizon 2020 research and innovation programme (grant agreement No.\ 850941). GM additionally gratefully acknowledges funding from the European Union’s Horizon Europe research and innovation programme under the Marie Skłodowska--Curie grant agreement No.\ 101064261.}
\begin{appendices}
	\section{Proof of Theorem~\ref{thm:low_regularity_convergence_thm_NLS}}\label{app:proof_of_convergence_thm_splitting_NLS}
	We recall the statement of Thm.~\ref{thm:low_regularity_convergence_thm_NLS}:
	{\begin{theorem}\label{thm_app:low_regularity_convergence_thm_NLS} Fix $s>1/2$ (corresponding to the norm in which we measure convergence) and $\gamma\in(0,1]$ (corresponding to the convergence order). If \eqref{eqn:NLS_in_hasimoto_transform} has a solution $u\in\mathcal{C}^0(0,T;H^{s+\gamma}(\mathbb{T}))$ then $\exists$ $C,h_0>0$ depending on $s,\gamma, \sup_{t\in[0,T]}\|u\|_{H^{s+\gamma}}$, such that for any $0< h<h_0$ we have
\begin{align*}
			\left\|u^m-u(t_m)\right\|_{H^s}\leq Ch^\gamma, \quad \forall\, 0\leq m\leq \left\lfloor\frac{T}{h}\right\rfloor,
		\end{align*}
  where $t_m=h m, m=0,\dots,\lfloor T/h\rfloor,$ and $u^m$ is computed using the method \eqref{eqn:symmetric_resonance_method}.
\end{theorem}}
For the proof of this statement it is useful to define the auxiliary function
\begin{align*}
    \mathcal{S}_1(u):=e^{ih {\partial_{x}^2}}u^m+i\frac{h}{4}e^{ih{\partial_{x}^2}}\left[\left(u^{m}\right)^2\varphi_1(-ih{\partial_{x}^2})\overline{u^m}\right]+i\frac{h}{4}\left[u^2\varphi_1(ih{\partial_{x}^2})\overline{u}\right].
\end{align*}
Then we can prove the following crucial result:
 
\begin{theorem}\label{thm:implicit_NLS_integrator_fixed_pt_iterations} Let $R>0$ and $s>1/2$. Then there is a $h_R>0$ such that for all $h\in [0,h_R)$ and any $u^{\corr{m}}\in B_R(H^s):=\{\tilde{v}\in H^s\,\vert\, \|\tilde{v}\|_{H^s}<R\}$ we have $u^{\corr{m+1}}$ the exact solution of \eqref{eqn:symmetric_resonance_method} is given by the following limit in $H^s$:
		\begin{align}\label{eqn:limit_expression_u^n+1}
			u^{m+1}=\lim_{j\rightarrow\infty}\mathcal{S}_1^{(j)}(e^{i{\partial_{x}^2} h}u^m),\quad\text{where\ \ }\mathcal{S}_1^{(j)}(u^m)=\underbrace{\mathcal{S}_1\circ\cdots\circ\mathcal{S}_1}_{j-\text{times}}(e^{ih{\partial_{x}^2}}u^m).
		\end{align}
		{Moreover, we have the estimate
			\begin{align}\label{eqn:distance_next_time_step_from_previous_Hs}
				\left\|u^{m+1}-e^{ih{\partial_{x}^2}}u^m\right\|_{H^s}\leq h \tilde{C}_R,
			\end{align}
			for some $\tilde{C}_R>0$ which depends only on $R$ (and $s$).}
\end{theorem}
\begin{proof} For the proof of \eqref{eqn:limit_expression_u^n+1} we notice that, for $h$ sufficiently small, $\mathcal{S}_1$ is a contraction mapping on $B_{R}(H^s)$: Indeed, for any $w,v\in B_{R}(H^s)$, we have
		\begin{align*}
			\left\|\mathcal{S}_1(w)-\mathcal{S}_1(v)\right\|_{H^s}=\frac{h}{4}\left\|w^2\varphi(ih{\partial_{x}^2})\bar{w}-v^2\varphi(ih{\partial_{x}^2})\bar{v}\right\|_{H^s}.
		\end{align*}
		Similarly to the proof of Lemma~3.1 in \cite{ostermann2018low} we now observe that there is a constant $c>0$ independent of $w,v$ such that
		\begin{align}\label{eqn:estimate_of_nonlinearity_NLS_convergence}
			\left\|w^2\varphi(ih{\partial_{x}^2})\bar{w}-v^2\varphi(ih{\partial_{x}^2})\bar{v}\right\|_{H^s}\leq c\left\|w-v\right\|_{H^l} \sum_{k=0}^{2}\|w\|_{H^s}^{2-k}\|v\|_{H^s}^k.
		\end{align}
  Thus if we take $h<6R^2c$ we find immediately that $\left\|\mathcal{S}_1(w)-\mathcal{S}_1(v)\right\|_{H^s}\leq\frac{1}{2}\left\|w-v\right\|_{H^s}$ and \eqref{eqn:limit_expression_u^n+1} follows from the Banach fixed-point theorem. Let us now write
\begin{align*}
    \|\mathcal{S}_1^{(J)}(e^{ih{\partial_{x}^2}}u^m)-e^{ih{\partial_{x}^2}}u^m\|_{H^s}&\leq \sum_{j=0}^{J-1}\|\mathcal{S}_1^{(j+1)}(e^{ih{\partial_{x}^2}}u^m)-\mathcal{S}_1^{(j)}(e^{ih{\partial_{x}^2}}u^m)\|_{H^s}\\
    &\leq \sum_{j=0}^{J-1}2^{-j}\|\mathcal{S}_1(e^{ih{\partial_{x}^2}}u^m)-e^{ih{\partial_{x}^2}}u^m\|_{H^s}\leq 2\|\mathcal{S}_1(e^{ih{\partial_{x}^2}}u^m)-e^{ih{\partial_{x}^2}}u^m\|_{H^s}
\end{align*}
Taking the limit as $J\rightarrow\infty$ it thus follows that
\begin{align*}
    \left\|u^{m+1}-e^{ih{\partial_{x}^2}}u^m\right\|_{H^s}\leq h\left\|\left(u^{m}\right)^2\varphi_1(-ih{\partial_{x}^2})\overline{u^m}\right\|_{H^s}\leq hc\|u^m\|_{H^s}^3
\end{align*}
  and so \eqref{eqn:distance_next_time_step_from_previous_Hs} follows.
\end{proof}
We can now prove the following additional properties of the symmetric low-regularity integrator from \eqref{eqn:symmetric_resonance_method} which we denote here by $\Phi_{h}^{NLS}$ (such that $u^{m+1}=\Phi_{h}^{NLS}(u^m)$):
\begin{lemma}[Stability]\label{lem:stability_NLS_integrator} Let $R>0, s>1/2$. Then there is a $h_R>0$ such that for all $0\leq h<h_R$ and any $v,w\in B_R(H^s)$ we have
\begin{align*}
\left\|\Phi_{h}^{NLS}(v)-\Phi_{h}^{NLS}(w)\right\|_{H^s}\leq \exp\left(C_Rh\right)\|v-w\|_{H^s}    
\end{align*}
where $C_R>0$ depends only on $R$ (and $s$).
\end{lemma}
\begin{proof} We have
    \begin{align*}
       \left\|\Phi_{h}^{NLS}(v)-\Phi_{h}^{NLS}(w)\right\|_{H^s}&\leq \left\|e^{ih{\partial_{x}^2}} (v-w)\right\|_{H^s}+\frac{h}{4}\left\|w^2\varphi(ih{\partial_{x}^2})\bar{w}-v^2\varphi(ih{\partial_{x}^2})\bar{v}\right\|_{H^s}\\
       &\quad\quad+\frac{h}{4}\left\|\left(\Phi_{h}^{NLS}(v)\right)^2\varphi(ih{\partial_{x}^2})\overline{\Phi_{h}^{NLS}(v)}-\left(\Phi_{h}^{NLS}(w)\right)^2\varphi(ih{\partial_{x}^2})\overline{\Phi_{h}^{NLS}(w)}\right\|_{H^s}\\
       &\leq \left(1+c\frac{h}{4}\sum_{j=0}^2\|u\|_{H^s}^j\|v\|_{H^s}^{2-j}\right)\|u-v\|_{H^s}\\
       &\quad\quad+c\frac{h}{4}\|\Phi_{h}^{NLS}(v)-\Phi_{h}^{NLS}(u)\|_{H^s}\sum_{j=0}^2\|\Phi_{h}^{NLS}(u)\|_{H^s}^j\|\Phi_{h}^{NLS}(v)\|_{H^s}^{2-j}\\
       &\leq \left(1+c\frac{h}{4}R^3\right)\|u-v\|_{H^s}+\frac{h}{4}(1+h_R\tilde{C}_R)^3\|\Phi_{h}^{NLS}(v)-\Phi_{h}^{NLS}(u)\|_{H^s}
    \end{align*}
    where in the final line we used Thm.~\ref{thm:implicit_NLS_integrator_fixed_pt_iterations}. Thus we have
    \begin{align*}
        \left\|\Phi_{h}^{NLS}(v)-\Phi_{h}^{NLS}(w)\right\|_{H^s}\leq \frac{1+h c_2}{1-h c_2}\left\|v-u\right\|_{H^s}\leq \exp(c_2h )\|v-u\|_{H^s}
    \end{align*}
    for some constant $c_2$ depending only on $R,s$ and the result follows.
\end{proof}

\begin{lemma}[Local error]\label{lem:local_error_NLS_integrator}Let us denote by $\varphi_{t}(u_0) = u(t)$ the exact solution to \eqref{eqn:NLS_in_hasimoto_transform} with initial condition $u_0$. Fix $R > 0,s>1/2,\gamma\in(0,1]$, then there is a $h_R> 0$ such that for all $h\in[0,h_R)$ and any $u_0\in B_R(H^s)$ with $\sup_{t\in[0,h ]}\|\varphi_{t}(u_0)\|_{H^{s+\gamma}}<R$ we have
\begin{align*}
    \left\|\varphi_{h}(u_0)-\Phi^{(NLS)}_h(u_0)\right\|_{H^s}\leq  c_Rh^{1+\gamma}
\end{align*}
for some constant $c_R > 0$ depending only on $R,s,\gamma > 0$.
\end{lemma}
\begin{proof} We observe that under the assumptions of this lemma the following estimate was proved in \cite[Lemma~3.2]{ostermann2018low}
\begin{align}\label{eqn:original_estimate_ostermann_schratz}
      \left\|\varphi_h(u_0)-e^{ih {\partial_{x}^2}}u_0-i\frac{h}{2}e^{ih{\partial_{x}^2}}\left[u_0^2\varphi_1(-2ih{\partial_{x}^2})\overline{u_0}\right]\right\|_{H^s}\leq c_{1,R}h^{1+\gamma},
\end{align}
for some constant $c_{1,R} > 0$ depending only on $R,s,\gamma > 0$. \corr{Taking two time steps of size $h/2$ in \eqref{eqn:original_estimate_ostermann_schratz} and iterating this error estimate (as well as using that $H^s$ is a Banach algebra) we find
\begin{align}\label{eqn:local_error_estimate_ostermann_schratz}
          \left\|\varphi_h(u_0)-e^{ih {\partial_{x}^2}}u_0-i\frac{h}{4}e^{ih{\partial_{x}^2}}\left[u_0^2\varphi_1(-ih{\partial_{x}^2})\overline{u_0}\right]-i\frac{h}{4}e^{i\frac{h}{2}{\partial_{x}^2}}\left[\left(e^{i\frac{h}{2}{\partial_{x}^2}}u_0\right)^2\varphi_1(-ih{\partial_{x}^2})\overline{e^{i\frac{h}{2}{\partial_{x}^2}}u_0}\right]\right\|_{H^s}\leq c_{2,R}h^{1+\gamma},
\end{align}
for some constant $c_{2,R} > 0$ depending only on $R,s,\gamma > 0$.} \corr{Moreover, we note by estimate \eqref{eqn:distance_next_time_step_from_previous_Hs} from Thm.~\ref{thm:implicit_NLS_integrator_fixed_pt_iterations} and since $H^s$ is a Banach algebra that, for $h$ sufficiently small, there is a constant $c_{3,R}$ depending on $R,s,\gamma > 0$ such that
\begin{align}\label{eqn:explicit_expression_symmetric_method}
    \left\|\Phi_{h}^{NLS}(u_0)-e^{ih {\partial_{x}^2}}u_0-i\frac{h}{4}e^{ih{\partial_{x}^2}}\left[u_0^2\varphi_1(-ih{\partial_{x}^2})\overline{u_0}\right]-i\frac{h}{4}\left[\left(e^{ih{\partial_{x}^2}}u_0\right)^2\varphi_1(ih{\partial_{x}^2})\overline{e^{ih{\partial_{x}^2}}u_0}\right]\right\|_{H^s}\leq c_{3,R} h^{2}.
\end{align}}
\corr{Finally, we have, writing $\hat{u}_k(0)$ for the Fourier coefficients of $u_0$
\begin{align}\begin{split}\label{eqn:explicit_fourier_mode_estimate}
    &\left\|i\frac{h}{4}e^{i\frac{h}{2}{\partial_{x}^2}}\left[\left(e^{i\frac{h}{2}{\partial_{x}^2}}u_0\right)^2\varphi_1(-ih{\partial_{x}^2})\overline{e^{i\frac{h}{2}{\partial_{x}^2}}u_0}\right]-i\frac{h}{4}\left[\left(e^{ih{\partial_{x}^2}}u_0\right)^2\varphi_1(ih{\partial_{x}^2})\overline{e^{ih{\partial_{x}^2}}u_0}\right]\right\|_{H^s}^2\\
    &\quad=\frac{h}{4}\sum_{k\in\mathbb{Z}}(1+|k|^{2s})\left|\sum_{k+l=m+n}e^{-i\frac{h}{2}(k^2-l^2+m^2+n^2)}\varphi_1(ihl^2)-e^{-ih(m^2+n^2-l^2)}\varphi_1(-ihl^2)\right|^2\left|\hat{u}_m(0)\hat{u}_n(0)\overline{\hat{u}_l(0)}\right|^2.
\end{split}\end{align}
Then, we have the central identity (valid for $k=-l+m+n$)
\begin{align*}
   e^{-i\frac{h}{2}k^2}\varphi_1(ihl^2)-e^{-i\frac{h}{2}(-l^2+m^2+n^2)}\varphi_1(-ihl^2)=e^{-i\frac{h}{2}k^2}\varphi_1(ihl^2)\left(1-e^{ih(mn-lm-ln)}\right),
\end{align*}
which implies, for any $\gamma\in (0,1]$ and any $m,n,k,l\in\mathbb{Z}$ with $k=-l+m+n$,
\begin{align*}
\left|e^{-i\frac{h}{2}k^2}\varphi_1(ihl^2)-e^{-i\frac{h}{2}(-l^2+m^2+n^2)}\varphi_1(-ihl^2)\right|&\leq \frac{\left|1-e^{ihmn-ihlm-ihln}\right|}{|h(mn-lm-ln)|^\gamma}|h(mn-lm-ln)|^\gamma\\
&\leq 2|h(mn-lm-ln)|^\gamma.
\end{align*}
Substituting this estimate into \eqref{eqn:explicit_fourier_mode_estimate} we deduce that (cf. proof of Lemma~3.2 in \cite{ostermann2018low})
\begin{align}\label{eqn:estimate_difference_implicit_explicit}
    \left\|i\frac{h}{4}e^{i\frac{h}{2}{\partial_{x}^2}}\left[\left(e^{i\frac{h}{2}{\partial_{x}^2}}u_0\right)^2\varphi_1(-ih{\partial_{x}^2})\overline{e^{i\frac{h}{2}{\partial_{x}^2}}u_0}\right]-i\frac{h}{4}\left[\left(e^{ih{\partial_{x}^2}}u_0\right)^2\varphi_1(ih{\partial_{x}^2})\overline{e^{ih{\partial_{x}^2}}u_0}\right]\right\|_{H^s}\leq c_{4,R}h^{1+r},
\end{align}
for a constant $c_{4,R}$ depending on $R,s,\gamma > 0$. We conclude by combining the estimates  \eqref{eqn:local_error_estimate_ostermann_schratz}, \eqref{eqn:explicit_expression_symmetric_method} \& \eqref{eqn:estimate_difference_implicit_explicit}.}
\end{proof}
\begin{proof}[Proof of Thm.~\ref{thm_app:low_regularity_convergence_thm_NLS}]
    The proof of this global convergence result is now a direct consequence of Lemmas \ref{lem:stability_NLS_integrator} \& \ref{lem:local_error_NLS_integrator} using a standard Lady Windermere’s fan argument (cf. proof of Thm 3.3 in \cite{ostermann2018low}).
\end{proof}
\section{Proof of Proposition~\ref{prop:quadrature_error_Q1}}\label{app:proof_quadrature_error_Q1}
We recall the statement of the proposition:
\begin{proposition} If $u$ is the solution to \eqref{eqn:NLS_in_hasimoto_transform} then
	\begin{align}\label{eqn:app_error_Q1}
	\left\|I_1[u,t_m,t_{m+1}]- \mathcal{Q}_1[u,t_m,t_{m+1}]\right\|_{L^2}\leq h^2 C,
	\end{align}
where $C$ depends on $\max_{t\in[t_m,t_{m+1}]}\|u(t)\|_{H^1}$. Moreover, for any $u,w:[t_m,t_{m+1}]\times \mathbb{T}\rightarrow\mathbb{C}$ we have
 \begin{align}\label{eqn:app_stability_Q1}
 	\left\|\mathcal{Q}_1[u,t_m,t_{m+1}]-\mathcal{Q}_1[w,t_m,t_{m+1}]\right\|_{L^2}\leq Dh\left(\left\|u(t_{m+1})-w(t_{m+1})\right\|_{H^1}+\left\|u(t_{m})-w(t_{m})\right\|_{H^1}\right),
 \end{align}
where $D>0$ is a constant independent of $u,w$.
\end{proposition}
\begin{proof}We begin by proving \eqref{eqn:app_error_Q1}: By construction of the quadrature we have
	\begin{align*}
		\left\|I_1[u,t_m,t_{m+1}]- \mathcal{Q}_1[u,t_m,t_{m+1}]\right\|_{L^2}&=\left\|\int_{t_m}^{t_{m+1}} e^{i{\partial_{x}^2} s}\partial_x\left(v(s)-\frac{v(t_{m+1})+v(t_m)}{2}\right)\dd s\right\|_{L^2}\\
		&\leq h\max_{s\in[t_m,t_{m+1}]}\left\|v(s)-\frac{v(t_{m+1})+v(t_m)}{2}\right\|_{H^1}.
	\end{align*}
	Finally, we note using \eqref{eqn:twisted_variable_equation} that
	\begin{align*}
		\left\|v(s)-\frac{v(t_{m+1})+v(t_m)}{2}\right\|_{H^1}\!\!\!\!\!\!&=\frac{1}{4}\left\|\int_{t_m}^s \!\!\!e^{-it{\partial_{x}^2}}\left[\left|e^{it{\partial_{x}^2}}v{(t)}\right|^2\!e^{it{\partial_{x}^2}}v{(t)}\right]\dd {t} -\int_{s}^{t_{m+1}}\!\!\!\!\!\!e^{-it{\partial_{x}^2}}\left[\left|e^{it{\partial_{x}^2}}v{(t)}\right|^2e^{it{\partial_{x}^2}}v{(t)}\right]\dd {t}\right\|_{H^1}\\
		&\leq \frac{h}{2}{\sup_{t\in[t_m,t_{m+1}]}}\left\|u{(t)}\right\|_{H^1}^3,
	\end{align*}
	from which \eqref{eqn:app_error_Q1} follows. To prove \eqref{eqn:app_stability_Q1} we let $z(t)=\exp(-it{\partial_{x}^2})w(t)$ such that
	\begin{align*}
	&\left\|\mathcal{Q}_1[u,t_m,t_{m+1}]-\mathcal{Q}_1[w,t_m,t_{m+1}]\right\|_{L^2}\\
	&\quad\quad\quad\quad\leq\frac{1}{2}\left\|\int_{t_m}^{t_{m+1}} e^{i{\partial_{x}^2} s}\partial_x\left(z(t_{m+1})-v(t_{m+1})\right)\dd s\right\|_{L^2}+\frac{1}{2}\left\|\int_{t_m}^{t_{m+1}} e^{i{\partial_{x}^2} s}\partial_x\left(z(t_{m})-v(t_{m})\right)\dd s\right\|_{L^2}\\
	&\quad\quad\quad\quad\leq\frac{h}{2}\|z(t_{m+1})-v(t_{m+1})\|_{L^2}+\frac{h}{2}\|z(t_{m})-v(t_{m})\|_{L^2},
	\end{align*}
from which \eqref{eqn:app_stability_Q1} follows.
\end{proof}
\section{Fast computation of index-restricted convolution-type sums}\label{app:fast_computation_of_index_restricted_convolutions}
Here we describe how we can exploit fast Toeplitz matrix vector products to compute the quadrature $\mathcal{Q}_2[u,t_m,t_{m+1}]$ on a spatial grid of size $N$ in just $\mathcal{O}(N(\log N)^2)$ operations. The algorithm is motivated by the work of \cite[Section 3]{hairer1985fast} and \cite[Section 4.1]{banjai2010multistep} on fast convolutional quadrature. We recall from section~\ref{sec:low_reg_hasimoto_transfrom} that our goal is the computation of 
\begin{align*}
	\mathcal{Q}_2[u,t_m,t_{m+1}](x_n)&=\sum_{l\in\mathbb{Z}}e^{il x_n}\left[\sum_{\substack{l=k_1-k_2\\k_1^2\geq k_2^2}}h\varphi_1(-ih k_1^2) \frac{e^{i h k_1^2}\hat{u}_{k_1}(t_{m+1})+\hat{u}_{k_1}(t_m)}{2} \frac{\overline{\hat{u}_{k_2}(t_{m+1})}+\overline{\hat{u}_{k_2}(t_m)}}{2}\right.\\
	&\quad\quad\quad\quad\quad\quad\quad\quad+\left.\sum_{\substack{l=k_1-k_2\\k_1^2< k_2^2}}h\varphi_1(ih k_2^2) \frac{\hat{u}_{k_1}(t_{m+1})+\hat{u}_{k_1}(t_m)}{2} \frac{\overline{e^{i h k_2^2}\hat{u}_{k_2}(t_{m+1})}+\overline{\hat{u}_{k_2}(t_m)}}{2}\right],
\end{align*}
for $n=-N/2+1,\dots, N/2$ where $x_n=2\pi\frac{n}{N}$, in $\mathcal{O}\left(N(\log N)^2\right)$ operations from given values of $\hat{v}_{k}(t_{m}),\hat{v}_{k}(t_{m+1}), k=-N/2+1,\dots N/2$. Let us focus on the first sum in this expression, as the second one can be computed analogously - in particular we aim to compute the values
\begin{align*}
	S_n:=\sum_{\substack{n=k_1-k_2\\k_1^2\geq k_2^2}}h\varphi_1(-ih k_1^2) \frac{e^{i h k_1^2}\hat{u}_{k_1}(t_{m+1})+\hat{u}_{k_1}(t_m)}{2} \frac{\overline{\hat{u}_{k_2}(t_{m+1})}+\overline{\hat{u}_{k_2}(t_m)}}{2},
\end{align*}
where we note that the corresponding contribution to $\mathcal{Q}_2[u,t_m,t_{m+1}](x_j), j=-N/2+1,\dots, N/2,$ from $S_n, n=-N/2+1,\dots, N/2,$ can be found in $\mathcal{O}(N\log N)$ operations by applying an FFT. Let us write
\begin{align*}
	w_{k}&:=h\varphi_1(-ih k^2) \frac{e^{i h k^2}\hat{u}_{k}(t_{m+1})+\hat{u}_{k}(t_m)}{2},\quad\text{and}\quad z_{j}:=\frac{\overline{\hat{u}_{-j}(t_{m+1})}+\overline{\hat{u}_{-j}(t_m)}}{2},
\end{align*}
then the computation of $S_n$ reduces to the following index-restricted convolution
\begin{align*}
S_n=\sum_{\substack{n=k+j\\k^2\geq j^2}}w_{k}z_{j}, \quad -N/2+1\leq {n}\leq N/2.
\end{align*}
All values of $w_k,z_j,$ for $k,j=-N/2+1,\dots, N/2$ can be computed in $\mathcal{O}(N)$ operations from $\hat{u}_{k}(t_{m}),\hat{u}_{k}(t_{m+1}),k=-N/2+1,\dots,N/2,$ by simply multiplying (and reversing) the list of values ${\hat{u}_{n}}, n=-N/2+1,\dots, N/2$. We note that if ${n}\geq 0$, then the condition $k^2\geq j^2$ is equivalent to $k\geq j$ and if ${n}<0$ the condition is equivalent to $k<j$. Thus
\begin{align*}
	S_n=\begin{cases}
		S_n^+:=\sum_{\substack{n=k+j\\k\geq j}}w_{k}z_{j},&{n}\geq 0,\\
		S_n^-:=\sum_{\substack{n=k+j\\k<j}}w_{k}z_{j},& {n}<0,
	\end{cases}
\end{align*}
and we can separately compute each vector $S_n^\pm, n=-N/2+1,\dots, N/2$ to recover $S_n$ for all $n=-N/2+1,\dots, N/2$ in $\mathcal{O}(N)$ operations. Let  us describe how to compute $S_n^+$, and note that the second sum $S_n^-$ can be computed analogously (or alternatively in $\mathcal{O}(N\log N)$ operations by subtracting $S_n^+$ from the full convolution $\sum_{n=k+j}w_{k}z_{j}$). We have
\begin{align*}
	S_n^+=\sum_{\substack{-N/2+1\leq k\leq N/2 \\k\geq n-k}} v_{n-k}w_{k}.
\end{align*}
This can be seen as a matrix vector multiplication of the vector $\mathbf{w}=(w_k)_{k=-N/2+1}^{N/2}$, by the matrix $\mathsf{V}$:
\begin{align*}
	\mathsf{V}_{kn}=\begin{cases}
		v_{n-k}, &\text{if\ } -N/2+1\leq n-k\leq N/2\text{\ and\ } n-k\leq k,\\
		0,&\text{otherwise}.
	\end{cases}
\end{align*}
We will focus on the case when $N=2^{N_0}, N_0\in\mathbb{N},$ and note that any other case of $N\in\mathbb{N}$ can be reduced to this by zero padding. An example for $N=16$ of this matrix is shown in Fig.~\ref{fig:sketch_of_first_matrix_subdivision}. Our methodology is now to partition the non-zero entries of the matrix $\mathsf{V}$ into Toeplitz blocks in a structured way, noting that the multiplication by a Toeplitz matrix of size $M\times M$ can be performed in $\mathcal{O}(M\log M)$ operations (for instance by embedding into a $2M\times 2M$ circulant matrix and using FFT).

In what follows we will describe an algorithm to partition the non-zero entries of $\mathsf{V}$ into Toeplitz blocks of the following sizes:
\begin{itemize}
    \item $1$ Toeplitz block, $\mathsf{A}^{(1,1)}$, of size $N/2\times N/2$,
    \item $2$ Toeplitz blocks, $\mathsf{A}^{(2,1)},\mathsf{A}^{(2,2)}$, of sizes $N/4\times N/4$,\ \\\vspace{-1cm}\\
    \begin{align*}
        \vdots\hspace{10cm}
    \end{align*}\ \\\vspace{-1.75cm}\\
    \item $N/4$ Toeplitz blocks, $\mathsf{A}^{(\log_2 N-1,j)},j=1,\dots N/4$, of sizes $2\times 2$,
    \item $N$ blocks, $\mathsf{A}^{(\log_2 N,j)},j=1,\dots N$, of size $1\times 1$.
\end{itemize}
Following this partition we can perform the matrix vector product $\mathbf{S}^+=V\mathbf{w}$ by using blockwise fast Toeplitz products in the following number of operations
\begin{align*}
\underbrace{N}_{(I)}+\sum_{j=1}^{\log_2 N-1}\underbrace{2^{j-1}}_{(II)}\underbrace{\mathcal{O}\left(\frac{N}{2^j}\log N\right)}_{(III)}+\underbrace{\frac{N}{2}}_{(IV)}=\mathcal{O}(N(\log N)^2),
\end{align*}
where $(I)$ corresponds to the cost of $N$ matrix multiplications of size $1\times 1$, $(II)$ corresponds to the corresponds to the number of matrices of size $\frac{N}{2^j}\times\frac{N}{2^j}$, $(III)$ corresponds to the cost of the fast Toeplitz product for a matrix of size $\frac{N}{2^j}\times\frac{N}{2^j}$, $(IV)$ and $(IV)$ corresponds to the cost of adding this contribution to the final output vector $\mathbf{S}^+$. It remains to show how we partition $\mathsf{V}$ into the Toeplitz blocks $\mathsf{A}^{(i,j)}$. We begin by defining the following three initial blocks:
\begin{align*}
    \mathsf{A}^{(1,1)}_{ij}&=V_{(i+N/2)j},\quad i,j=1,\dots, N/2,\\
    \mathsf{A}^{(2,1)}_{ij}&=V_{ij},\quad i,j=1,\dots, N/4,\\
    \mathsf{A}^{(2,2)}_{ij}&=V_{(i+N/4)j},\quad i,j=1,\dots, N/4.
\end{align*}
The remaining nonzero entries of $\mathsf{V}$ are contained in two rectangular blocks $\mathsf{B}^{(2,1)},\mathsf{B}^{(2,2)}$ given by
\begin{align*}
    \mathsf{B}^{(2,1)}_{ij}&=\mathsf{V}_{i(j+N/4)},\quad i=1,\dots,N/2,j=1,\dots, N/4,\\
    \mathsf{B}^{(2,1)}_{ij}&=\mathsf{V}_{(i+N/2)(j+3N/4)},\quad i=1,\dots,N/2,j=1,\dots, N/4.
\end{align*}
This partition is shown in Fig.~\ref{fig:sketch_of_first_matrix_subdivision}. The matrices $\mathsf{B}^{(2,1)},\mathsf{B}^{(2,2)}$ are of the following general form.
\begin{definition} We say a $2^{k+1}\times 2^{k}$-matrix $\mathsf{B}$, some $k\in\mathbb{N}$, is of type M if it has entries of the form
\begin{align*}
    \mathsf{B}_{ij}=\begin{cases}
        c_{i-j},&\text{if\ }i\geq 2j,\\
        0,&\text{otherwise}.
    \end{cases}
\end{align*}
for some $c_{j}\in\mathbb{C},j=1,\dots 2^{k+1}-1$.
\end{definition}

\begin{figure}[h!]\centering
\begin{subfigure}{0.7\textwidth}
\begin{tikzpicture}[>=stealth,thick,baseline]
		\matrix [matrix of math nodes,left delimiter=(,right delimiter=)](A){ 
			v_{0} &v_{-1}&v_{-2}&v_{-3}&0&0&0&0&0&0&0&0&0&0&0&0\\
			v_{1} &v_{0}&v_{-1}&v_{-2}&v_{-3}&0&0&0&0&0&0&0&0&0&0&0\\
			v_{2} &v_{1}&v_{0}&v_{-1}&v_{-2}&0&0&0&0&0&0&0&0&0&0&0\\
			v_{3} &v_{2}&v_{1}&v_{0}&v_{-1}&v_{-2}&0&0&0&0&0&0&0&0&0&0\\
			v_{4} &v_{3}&v_{2}&v_{1}&v_{0}&v_{-1}&0&0&0&0&0&0&0&0&0&0\\
			v_{5} &v_{4}&v_{3}&v_{2}&v_{1}&v_{0}&v_{-1}&0&0&0&0&0&0&0&0&0\\
			v_{6} & v_{5} &v_{4}&v_{3}&v_{2}&v_{1}&v_{0}&0&0&0&0&0&0&0&0&0\\
			v_{7} & v_{6} & v_{5} &v_{4}&v_{3}&v_{2}&v_{1}&v_{0}&0&0&0&0&0&0&0&0\\
			v_{8} & v_{7} & v_{6} &v_{5}&v_{4}&v_{3}&v_{2}&v_{1}&0&0&0&0&0&0&0&0\\
			0 & v_{8} & v_{7} &v_{6}&v_{5}&v_{4}&v_{3}&v_{2}&v_{1}&0&0&0&0&0&0&0\\
			0& 0 & v_{8} & v_{7} &v_{6}&v_{5}&v_{4}&v_{3}&v_{2}&0&0&0&0&0&0&0\\
			0& 0& 0 & v_{8} & v_{7} &v_{6}&v_{5}&v_{4}&v_{3}&v_{2}&0&0&0&0&0&0\\
			0& 0& 0& 0 & v_{8} & v_{7} &v_{6}&v_{5}&v_{4}&v_{3}&0&0&0&0&0&0\\
			0&0& 0& 0& 0 & v_{8} & v_{7} &v_{6}&v_{5}&v_{4}&v_{3}&0&0&0&0&0\\
			0&0&0& 0& 0& 0 & v_{8} & v_{7} &v_{6}&v_{5}&v_{4}&0&0&0&0&0\\
			0&0&0&0& 0& 0& 0 & v_{8} & v_{7} &v_{6}&v_{5}&v_{4}&0&0&0&0\\
		};
	
	\node[rectangle,draw,thick,line width=2pt,
	text = olive,
	fill opacity=0.5,
	fill = green!10!white,
	minimum width = 5.75cm, 
	minimum height = 4.0cm,
	yshift=-0.25cm,xshift=0.4cm] (r) at (A-12-4) {};

	\node[rectangle,draw,dashed,line width=2pt,
	text = olive,
	fill opacity=0.5,
	fill = blue!10!white,
	minimum width = 2.8cm, 
	minimum height = 2.1cm,
	yshift=-0.2cm,xshift=0.45cm] (r) at (A-6-2) {};

	\node[rectangle,draw,dashed,line width=2pt,
	text = olive,
	fill opacity=0.5,
	fill = blue!10!white,
	minimum width = 2.8cm, 
	minimum height = 2.1cm,
	yshift=-0.25cm,xshift=0.45cm] (r) at (A-2-2) {};

 	\node[rectangle,draw, dashdotted,line width=2pt,
	text = olive,
	fill opacity=0.5,
	fill = red!10!white,
	minimum width = 2.9cm, 
	minimum height = 4.2cm,
	yshift=-0.2cm,xshift=0.3cm] (r) at (A-4-6) {};

 	\node[rectangle,draw, dashdotted,line width=2pt,
	text = olive,
	fill opacity=0.5,
	fill = red!10!white,
	minimum width = 2.2cm, 
	minimum height = 4.05cm,
	yshift=-0.22cm,xshift=0.3cm] (r) at (A-12-10) {};

	\matrix [matrix of math nodes,left delimiter=(,right delimiter=)](B){ 
	v_{0} &v_{-1}&v_{-2}&v_{-3}&0&0&0&0&0&0&0&0&0&0&0&0\\
	v_{1} &v_{0}&v_{-1}&v_{-2}&v_{-3}&0&0&0&0&0&0&0&0&0&0&0\\
	v_{2} &v_{1}&v_{0}&v_{-1}&v_{-2}&0&0&0&0&0&0&0&0&0&0&0\\
	v_{3} &v_{2}&v_{1}&v_{0}&v_{-1}&v_{-2}&0&0&0&0&0&0&0&0&0&0\\
	v_{4} &v_{3}&v_{2}&v_{1}&v_{0}&v_{-1}&0&0&0&0&0&0&0&0&0&0\\
	v_{5} &v_{4}&v_{3}&v_{2}&v_{1}&v_{0}&v_{-1}&0&0&0&0&0&0&0&0&0\\
	v_{6} & v_{5} &v_{4}&v_{3}&v_{2}&v_{1}&v_{0}&0&0&0&0&0&0&0&0&0\\
	v_{7} & v_{6} & v_{5} &v_{4}&v_{3}&v_{2}&v_{1}&v_{0}&0&0&0&0&0&0&0&0\\
	v_{8} & v_{7} & v_{6} &v_{5}&v_{4}&v_{3}&v_{2}&v_{1}&0&0&0&0&0&0&0&0\\
	0 & v_{8} & v_{7} &v_{6}&v_{5}&v_{4}&v_{3}&v_{2}&v_{1}&0&0&0&0&0&0&0\\
	0& 0 & v_{8} & v_{7} &v_{6}&v_{5}&v_{4}&v_{3}&v_{2}&0&0&0&0&0&0&0\\
	0& 0& 0 & v_{8} & v_{7} &v_{6}&v_{5}&v_{4}&v_{3}&v_{2}&0&0&0&0&0&0\\
	0& 0& 0& 0 & v_{8} & v_{7} &v_{6}&v_{5}&v_{4}&v_{3}&0&0&0&0&0&0\\
	0&0& 0& 0& 0 & v_{8} & v_{7} &v_{6}&v_{5}&v_{4}&v_{3}&0&0&0&0&0\\
	0&0&0& 0& 0& 0 & v_{8} & v_{7} &v_{6}&v_{5}&v_{4}&0&0&0&0&0\\
	0&0&0&0& 0& 0& 0 & v_{8} & v_{7} &v_{6}&v_{5}&v_{4}&0&0&0&0\\
};
	\end{tikzpicture}
		\caption{The matrix $\mathsf{V}$.}
  \label{fig:sketch_of_first_matrix_subdivision}
 \end{subfigure}
\begin{subfigure}{0.2\textwidth}
	\begin{tikzpicture}[>=stealth,thick,baseline]
		\matrix [matrix of math nodes,left delimiter=(,right delimiter=)](A){ 
			0&0&0&0\\
			c_{1} &0&0&0\\
			c_{2} &0&0&0\\
			c_{3} &c_{2}&0&0\\
			c_{4} & c_{3}&0&0\\
			c_{5} &c_{4}&c_{3}&0\\
			c_{6} & c_{5} &c_{4}&0\\
			c_{7} & c_{6} & c_{5} &c_{4}\\
		};
	
	\node[rectangle,draw,thick,line width=2pt,
	text = olive,
	fill opacity=0.5,
	fill = green!10!white,
	minimum width = 1.1cm, 
	minimum height = 1.0cm,
	yshift=-0.3cm,xshift=0.255cm] (r) at (A-7-1) {};

	\node[rectangle,draw,thick,line width=2pt,
	text = olive,
	fill opacity=0.5,
	fill = green!10!white,
	minimum width = 1.1cm, 
	minimum height = 1.0cm,
	yshift=-0.3cm,xshift=0.255cm] (r) at (A-5-1) {};

	\node[rectangle,draw,dashed,thick,line width=2pt,
	text = olive,
	fill opacity=0.5,
	fill = blue!10!white,
	minimum width = 1.1cm, 
	minimum height = 2.05cm,
	yshift=-0.25cm,xshift=0.255cm] (r) at (A-2-1) {};

	\node[rectangle,draw,dashed,thick,line width=2pt,
	text = olive,
	fill opacity=0.5,
	fill = blue!10!white,
	minimum width = 1.1cm, 
	minimum height = 2.05cm,
	yshift=-0.265cm,xshift=0.29cm] (r) at (A-6-3) {};

	\matrix [matrix of math nodes,left delimiter=(,right delimiter=)](B){ 
			0&0&0&0\\
			c_{1} &0&0&0\\
			c_{2} &0&0&0\\
			c_{3} &c_{2}&0&0\\
			c_{4} & c_{3}&0&0\\
			c_{5} &c_{4}&c_{3}&0\\
			c_{6} & c_{5} &c_{4}&0\\
			c_{7} & c_{6} & c_{5} &c_{4}\\
		};
		
	\end{tikzpicture}
 \caption{The matrix $\mathsf{B}$.}
\label{fig:example_type_M_decomposition}
 \end{subfigure}
\caption{(a) The matrix $\mathsf{V}$ and its partitioning into $\mathsf{A}^{(1,1)}$ (green with solid boundary), $\mathsf{A}^{(2,1)},\mathsf{A}^{(2,2)}$ (blue with dashed boundary) and $\mathsf{B}^{(2,1)},\mathsf{B}^{(2,2)}$ (red with dashdotted boundary). (b) The matrix $\mathsf{B}$ and its partition into $\mathsf{A}^{(1)},\mathsf{A}^{(2)}$ (green with solid boundary) and $\mathsf{B}^{(1)},\mathsf{B}^{(2)}$ (blue with dashed boundary).}
\end{figure}
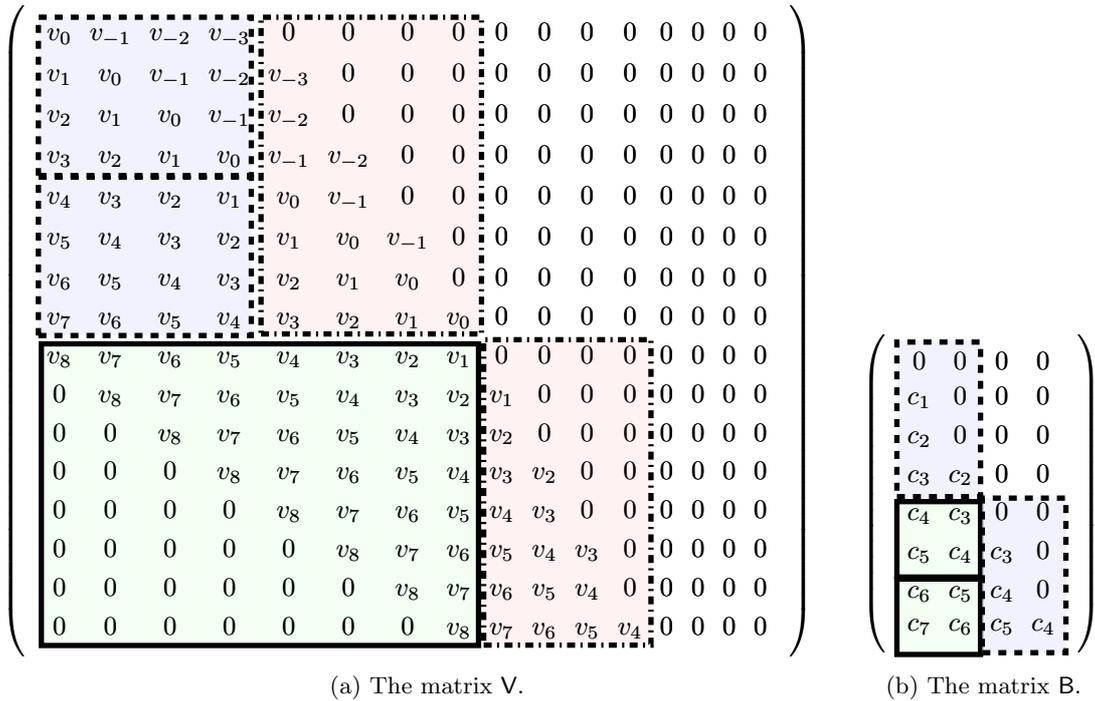
The nonzero entries of each matrix $\mathsf{B}$ of type M with size $2^{k+1}\times 2^{k}$ can be partitioned into two Toeplitz blocks $\mathsf{A}^{(1)},\mathsf{A}^{(2)}$ of size $2^{k-1}\times 2^{k-1}$ and two matrices $\mathsf{B}^{(1)}$ and $\mathsf{B}^{(2)}$ which are both of type M with size $2^{k}\times 2^{k-1}$, by taking
\begin{align*}
    \mathsf{A}^{(1)}_{ij}&=\mathsf{B}_{(i+2^{k})j},\quad i,j=1,\dots, 2^{k-1},\\
    \mathsf{A}^{(2)}_{ij}&=\mathsf{B}_{(i+2^{k}+2^{k-1})j},\quad i,j=1,\dots, 2^{k-1},\\
    \mathsf{B}^{(1)}_{ij}&=\mathsf{B}_{ij},\quad i=1,\dots, 2^{k}, j=1,\dots, 2^{k-1},\\
    \mathsf{B}^{(2)}_{ij}&=\mathsf{B}_{(i+2^{k})(j+2^{k-1})},\quad i=1,\dots, 2^{k}, j=1,\dots, 2^{k-1}.\\
\end{align*}
An example for the case $k=2$ is shown in Fig.~\ref{fig:example_type_M_decomposition}. We can repeat this process inductively until we have partitioned the non-zero entries of $\mathsf{B}$ into Toeplitz blocks and the remaining type M blocks are of size $2\times 1$, which can in turn be partitioned into two $1\times 1$ matrices each. Applying this process to $\mathsf{B}^{(2,1)}, \mathsf{B}^{(2,2)}$ yields exactly the desired partition of our matrix $\mathsf{V}$.

\section{Proof of Proposition~\ref{prop:quadrature_error_Q2}}\label{app:proof_of_error_Q2}
We recall the statement of the result:
\begin{proposition}
	If $u$ is the solution to \eqref{eqn:NLS_in_hasimoto_transform} then
	\begin{align}\label{eqn:app_error_Q2}
		\left\|I_2[u,t_m,t_{m+1}]- \mathcal{Q}_2[u,t_m,t_{m+1}]\right\|_{L^2}\leq h^2 C,
	\end{align}
	where $C>0$ depends on $\max_{t\in[t_m,t_{m+1}]}\|u(t)\|_{H^1}$. Moreover, for any two functions $u,w:[t_m,t_{m+1}]\times \mathbb{T}\rightarrow\mathbb{C}$ we have:
	\begin{align}\label{eqn:app_stability_Q2}
		\left\|\mathcal{Q}_2[u,t_m,t_{m+1}]- \mathcal{Q}_2[w,t_m,t_{m+1}]\right\|_{L^2}\leq D h \left(\left\|u(t_{m+1})-w(t_{m+1})\right\|_{H^1}+\left\|u(t_{m})-w(t_{m})\right\|_{H^1}\right),
	\end{align}
	where $D>0$ depends on $\max\left\{\|u(t_m)\|_{H^1}+\|w(t_m)\|_{H^1},\|u(t_{m+1})\|_{H^1}+\|w(t_m)\|_{H^1}\right\}$.
\end{proposition}
\begin{proof} We begin by proving \eqref{eqn:app_error_Q2}: Note that
	\begin{align*}
		\hspace{-0.4cm}\left\|I_2[u,t_m,t_{m+1}]- \mathcal{Q}_2[u,t_m,t_{m+1}]\right\|_{L^2}^2
	&=\sum_{l\in\mathbb{Z}}\left|\sum_{l=k_1-k_2}\int_{t_{m}}^{t_{m+1}}e^{-is(k_1^2-k_2^2)} \hat{v}_{k_1}(s)\overline{\hat{v}_{k_2}(s)}\dd s\right.\\&\left.-\left[\sum_{\substack{l=k_1-k_2\\k_1^2\geq k_2^2}}\int_{0}^{h}e^{-isk_1^2}\dd s \frac{e^{i h k_1^2}\hat{u}_{k_1}(t_{m+1})+\hat{u}_{k_1}(t_m)}{2} \frac{\overline{\hat{u}_{k_2}(t_{m+1})}+\overline{\hat{u}_{k_2}(t_m)}}{2}\right.\right.\\
  &+\left.\left.\sum_{\substack{l=k_1-k_2\\k_1^2< k_2^2}}\int_{0}^{h}e^{isk_2^2}\dd s \frac{\hat{u}_{k_1}(t_{m+1})+\hat{u}_{k_1}(t_m)}{2} \frac{\overline{e^{i h k_2^2}\hat{u}_{k_2}(t_{m+1})}+\overline{\hat{u}_{k_2}(t_m)}}{2}\right]\right|^2.
	\end{align*}
	Now we note that 
	\begin{align*}
 &\left|\sum_{\substack{l=k_1-k_2\\k_1^2\geq k_2^2}}\left(\int_{t_{m}}^{t_{m+1}}e^{-is(k_1^2-k_2^2)} \hat{v}_{k_1}(s)\overline{\hat{v}_{k_2}(s)}\dd s-\int_{0}^{h}e^{-isk_1^2}\dd s \frac{e^{i h k_1^2}\hat{u}_{k_1}(t_{m+1})+\hat{u}_{k_1}(t_m)}{2} \frac{\overline{\hat{u}_{k_2}(t_{m+1})}+\overline{\hat{u}_{k_2}(t_m)}}{2}\right)\right|\\
		&\quad\leq \left|\sum_{\substack{l=k_1-k_2\\k_1^2\geq k_2^2}}e^{-it_m(k_1^2-k_2^2)}\int_{0}^{h}\left[e^{-is(k_1^2-k_2^2)}-e^{-isk_1^2}\right] \hat{v}_{k_1}(t_m+s)\overline{\hat{v}_{k_2}(t_m+s)}\dd s\right|\\
		&\quad\quad+\left|\sum_{\substack{l=k_1-k_2\\k_1^2\geq k_2^2}} \int_{0}^{h}e^{-isk_1^2}\left[e^{i s k_1^2}\hat{u}_{k_1}(t_m+s)\overline{e^{i s k_2^2}\hat{u}_{k_2}(t_m+s)}\right.\right.\\
  &\quad\quad\quad\quad\quad\quad\quad\quad\quad\quad\quad\quad\quad\quad\quad\quad\left.\left.-\frac{e^{i h k_1^2}\hat{u}_{k_1}(t_{m+1})+\hat{u}_{k_1}(t_m)}{2} \frac{\overline{\hat{u}_{k_2}(t_{m+1})}+\overline{\hat{u}_{k_2}(t_m)}}{2}\right]\dd s\right|
	\end{align*}
Moreover
	\begin{align*}
		&\left|\sum_{\substack{l=k_1-k_2\\k_1^2\geq k_2^2}}e^{-it_m(k_1^2-k_2^2)}\int_{0}^{h}\left[e^{-is(k_1^2-k_2^2)}-e^{-isk_1^2}\right] \hat{v}_{k_1}(t_m+s)\overline{\hat{v}_{k_2}(t_m+s)}\dd s\right|\\
  &\quad\quad\quad\quad\quad\quad\quad\quad\quad\quad\quad\quad\leq\sum_{\substack{l=k_1-k_2\\k_1^2\geq k_2^2}}\int_{t_{m}}^{t_{m+1}}s\left|k_2\right|^2 \left|\hat{u}_{k_1}\right|\left|\overline{\hat{u}_{k_2}}\right|\dd s\leq h^2\sup_{t\in[t_m,t_{m+1}]}\sum_{\substack{l=k_1-k_2\\k_1^2\geq k_2^2}} \left|k_1\hat{u}_{k_1}(t)\right|\left|k_2\overline{\hat{u}_{k_2}}(t)\right|\\
		&\quad\quad\quad\quad\quad\quad\quad\quad\quad\quad\quad\quad\leq h^2\sup_{t\in[t_m,t_{m+1}]}\sum_{\substack{l=k_1-k_2}} \left|\widehat{\partial_xu}_{k_1}(t)\right|\left|\widehat{\partial_xu}_{k_2}(t)\right|\leq ch^2 \sup_{t\in[t_m,t_{m+1}]}\|u(t)\|_{H^1}^2.
	\end{align*}
	A similar estimate holds true for the terms involving $\sum_{\substack{l=k_1-k_2\\k_1^2< k_2^2}}$ and so we find that
	\begin{align*}
		&\left\|I_2[u,t_m,t_{m+1}]- \mathcal{Q}_2[u,t_m,t_{m+1}]\right\|_{L^2}^2\\&\quad\quad\quad\lesssim h^2\sup_{t\in[t_m,t_{m+1}]}\|v(t)\|_{H^1}^2\\
  &\quad\quad\quad\quad+\left\|\int_{0}^{h}\left(e^{i s{\partial_{x}^2}}\left(u(t_{m}+s)-\frac{e^{-i h{\partial_{x}^2}}u(t_{m+1})+u(t_m)}{2}\right)\right)\overline{\left(u(t_m+s)-\frac{u(t_{m+1})+u(t_m)}{2}\right)}\dd s\right\|_{H^1}^2\\
		&\quad\quad\quad\quad\quad\quad+\left\|\int_{0}^{h}\left(u(t_m+s)-\frac{u(t_{m+1})+u(t_m)}{2}\right)\overline{\left(e^{i s{\partial_{x}^2}}\left(u(t_{m}+s)-\frac{e^{-i h{\partial_{x}^2}}u(t_{m+1})+u(t_m)}{2}\right)\right)}\dd s\right\|_{H^1}^2\end{align*}
and thus $\left\|I_2[u,t_m,t_{m+1}]- \mathcal{Q}_2[u,t_m,t_{m+1}]\right\|_{L^2}^2\lesssim h^2\sup_{t\in[t_m,t_{m+1}]}\|u(t)\|_{H^{1}}^2$	using similar estimates to the proof of Prop.~\ref{prop:quadrature_error_Q1}.
 
 To prove \eqref{eqn:app_stability_Q2} we let again $z(t)=\exp(-it{\partial_{x}^2})w(t)$ and observe that by the definition of $\varphi_1,\varphi_2$ we have 
 \begin{align*}
     &\left\|\mathcal{Q}_2[u,t_m,t_{m+1}]- \mathcal{Q}_2[w,t_m,t_{m+1}]\right\|_{L^2}\\
     &\quad\quad\leq h D\left\|\left(e^{-i h{\partial_{x}^2}}u(t_{m+1})+u(t_m)\right)\overline{\left(u(t_{m+1})+u(t_m)\right)}-\left(e^{-i h{\partial_{x}^2}}w(t_{m+1})+w(t_m)\right)\overline{\left(w(t_{m+1})+w(t_m)\right)}\right\|_{H^1}
 \end{align*}
 for some constant $D>0$ independent of $w,u,t_m,h$. The result then follows immediately by submultiplicativity of $\|\cdot\|_{H^1}$.
\end{proof}
\end{appendices}

\addcontentsline{toc}{section}{References}
\bibliographystyle{siam}     
\nocite{*}  
\bibliography{biblio1}

\end{document}